\documentclass{amsart}
\usepackage{amsmath,amsfonts,amssymb,amscd,amsthm,amsbsy,epsf,calc,graphicx,epsf}

\newtheorem{thm}{Theorem}[section]
\newtheorem{lem}[thm]{Lemma}

\newtheorem{cor}[thm]{Corollary}
\newtheorem{prop}[thm]{Proposition}
\newtheorem{ex}[thm]{Example}

\newtheorem{DEF}[thm]{Definition}
\theoremstyle{remark}                  
\newtheorem{rem}[thm]{Remark}

\newcommand{\Hn}{{\mathcal H}^n}
\newcommand{\Hnm}{{\mathcal H}^{n-1}}
\newcommand{\HM}{{\mathcal H}}
\newcommand{\hyperbolic}{{\mathbf H}}
\newcommand{\cost}{{\rm cost}\thinspace}
\newcommand{\dom}{{\rm dom}\thinspace}
\newcommand{\Prob}{{\mathcal P}}
\newcommand{\R}{{\mathbf R}}
\newcommand{\spt}{\mathop{\rm spt}\thinspace}
\newcommand{\sphere}{{\mathbf S}}
\newcommand{\torus}{{\mathbf T}}
\newcommand{\tb}{{\tilde b}}
\newcommand{\tc}{{\tilde c}}

\newcommand{\Azero}{{\rm (A0)}}
\newcommand{\Aone}{{\rm (A1)}}
\newcommand{\Atwo}{{\rm (A2)}}
\newcommand{\Athree}{{\rm (A3)}}
\newcommand{\Afour}{{\rm (A4)}}
\newcommand{\Athrees}{{\rm (A3)$_s$}}

\newcommand{\Bthree}{{\rm (B3)}}

\newcommand{\cross}{\mbox{\rm cross}}

\begin{document}

\title[Optimal transportation: geometry, regularity and applications]{Five lectures on optimal transportation: geometry, regularity and applications}
\author{Robert J. McCann$^*$ and Nestor Guillen}
\address{Department of Mathematics, University of Toronto, Toronto Ontario Canada M5S 2E4}
\email{mccann@math.toronto.edu}
\address{Department of Mathematics, University of Texas at Austin, Austin TX USA 78712}
\email{nguillen@math.utexas.edu}
\thanks{$^*$ [RJM]'s research was supported in part by grant 217006-08 of the Natural
Sciences and Engineering Research Council of Canada.}

\begin{abstract}
In this series of lectures we introduce the Monge-Kantorovich problem
of optimally transporting one distribution of mass onto another,
where optimality is measured against a cost function $c(x,y)$.
Connections to geometry, inequalities, and partial differential
equations will be discussed,  focusing in particular on recent developments
in the regularity theory for Monge-Amp\`ere type equations.
An application to microeconomics will also be described, which amounts to
finding the equilibrium price distribution for a
monopolist marketing a multidimensional line of products to a population of
anonymous agents whose preferences are known only statistically.
\hfill \copyright 2010 by Robert J. McCann. All rights reserved.
\end{abstract}

\maketitle
\tableofcontents


\section*{Preamble}

This survey is based on a series of five lectures
by Robert McCann (of the University of Toronto), delivered at a
summer school on
``New Vistas in Image Processing and Partial Differential Equations''
organized 7-12 June 2010 by Irene Fonseca, Giovanni Leoni, and Dejan Slepcev
of Carnegie Mellon University on
behalf of the Center for Nonlinear Analysis there.
The starting point for the manuscript which emerged was a detailed set of notes taken
during those lectures by Nestor Guillen (University of Texas at Austin).

These notes are intended to convey a flavor for the subject,  without getting bogged
down in too many technical details.  Part of the discussion is therefore impressionistic,
and some of the results are stated under
the tacit requirement that the supports of the measures $\mu^\pm$ be compact,
with the understanding that they extend to non-compactly supported measures under appropriate
hypotheses \cite{McCann95} \cite{GangboMcCann96} \cite{FathiFigalli10} \cite{FigalliGigli10p}
concerning the behaviour near infinity of the measures and the costs.
The choice of topics to be covered in a
series of lectures is necessarily idiosyncratic.
General references for these and other topics include papers of the first author posted
on the website {\tt www.math.toronto.edu/mccann} and the two books by
Villani \cite{Villani03} \cite{Villani09}.  Earlier surveys include the
ones by Ambrosio \cite{Ambrosio03}, Evans \cite{Evans98}, Urbas \cite{Urbas98p}
and Rachev and R\"uschendorf \cite{RachevRuschendorf98}.
Many detailed references to the literature may be found there,  to augment the
bibliography of selected works included below. \\

\section{An introduction to optimal transportation}

\subsection{Monge-Kantorovich problem: transporting ore from mines to factories}

The problem to be discussed can be caricatured as follows:
imagine we have a distribution of iron mines across the countryside,
producing a total of 1000 tonnes of iron ore weekly, and a distribution
of factories across the countryside that consume a total of 1000 tonnes of iron ore weekly.
Knowing the {\em cost} $c(x,y)$ per ton of ore transported from a mine at $x$ to a factory at $y$,
the problem is to decide which mines should be supplying which factories
so as to minimize the total transportation costs. \\

To model this problem mathematically,  let the triples $(M^{\pm},d^{\pm},\omega^\pm)$
denote two complete separable metric spaces $M^\pm$ --- also called {\em Polish} spaces ---
equipped with distance functions
$d^{\pm}$ and Borel reference measures $\omega^\pm$.  These two metric spaces will represent the
landscapes containing the mines and the factories. They will often be assumed to be
geodesic spaces,  and/or to coincide.  Here a
{\em geodesic space} $M(=M^\pm)$ refers to a metric space in which
every pair of points $x_0,x_1 \in M$ is connected by a curve
$s \in [0,1] \to x_s \in M$  satisfying
\begin{equation}\label{geodesic}
d(x_0,x_s)=sd(x_0,x_1) \qquad {\rm and} \qquad d(x_s,x_1)=(1-s)d(x_0,x_1)
\qquad \forall\ s \in [0,1].
\end{equation}
Such a curve is called a {\em geodesic segment}.

e.g. 1) Euclidean space: $M=\R^n$, $d(x,y)=|x-y|$, $\omega = Vol = \Hn =$ Hausdorff $n$-dimensional measure,
geodesic segments take the form $x_s=(1-s)x_0+sx_1$.\\

e.g. 2) Complete Riemannian manifold $(M=M^{\pm},g_{ij})$, with or without boundary:
$d\omega = dVol = d\Hn = (\det g_{ij})^{1/2} d^nx$,
$$\frac{1}{2} d^2(x_0,x_1)=\inf \limits_{\{x_s|x(0)=x_0,x(1)=x_1\}}
\frac{1}{2} \int_0^1 \langle\dot{x_s}, \dot{x_s}\rangle_{g(x_s)}ds.$$
A minimizing curve $s\in [0,1] \longmapsto x_s \in M$ exists by the Hopf-Rinow theorem;
it satisfies \eqref{geodesic}, and is called a Riemannian {\em geodesic}.

The distributions of mines and factories will be modeled by
Borel probability measures $\mu^+$ on $M^+$ and $\mu^-$ on $M^-$, respectively.
Any Borel map $G:M^+\longrightarrow M^-$ defines an image or {\em pushed-forward} measure
$\nu = G_\# \mu^+$ on $M^-$ by
\begin{equation}\label{push-forward} 
(G_\#\mu^+)[V] := \mu^+[G^{-1}(V)] \qquad \forall\ V \subset M^-.
\end{equation}
A central problem in optimal transportation is to find,
among all maps $G:M^+ \longrightarrow M^-$ pushing $\mu^+$ forward to $\mu^-$,
one which minimizes the total cost
\begin{equation}\label{Monge cost}
\cost(G)=\int_{M^+}c(x,G(x))d\mu^+(x).
\end{equation}

This problem was first proposed by Monge in 1781, taking
the Euclidean distance $c(x,y)=|x-y|$ as his cost function \cite{Monge81}.
For more generic costs, some
basic mathematical issues such as the existence, uniqueness, and mathematical structure
of the optimizers are addressed in the second lecture below.  However nonlinearity
of the objective functional and a lack of compactness or convexity for its domain
make Monge's formulation of the problem difficult to work with. One and a half centuries later,
Kantorovich's relaxation of the problem to an (infinite-dimesional) linear program
provided a revolutionary tool \cite{Kantorovich42} \cite{Kantorovich48}.

\subsection{Wasserstein distance and geometric applications}

The minimal cost of transport between $\mu^+$ and $\mu^-$ associated to $c(\cdot,\cdot)$
will be provisionally denoted by
\begin{equation}\label{Monge work}
W_c(\mu^+,\mu^-) = \inf \limits_{G_\#\mu^+=\mu^-} \cost(G).
\end{equation}
It can be thought of as quantifying the discrepancy between $\mu^+$ and $\mu^-$,
and is more properly defined using Kantorovich's formulation \eqref{Kantorovich},
though we shall eventually show the two definitions coincide in many cases of interest.
When $M=M^\pm$, the costs  $c(x,y)=d^p(x,y)$ with $0<p<1$ occur naturally
in economics and operations research,  where it is often the case that there is
an economy of scale for long trips \cite{McCann99}.  In this case, the quantity
$W_c(\mu^+,\mu^-)$ defines a metric on the space $\Prob(M)$ of Borel probability measures on $M$.
For $p \ge 1$ on the other hand,
it is necessary to extract a $p$-$th$ root to obtain a metric
\begin{equation}\label{Wasserstein}
d_p(\mu^+,\mu^-) := W_c(\mu^+,\,u^-)^{1/p}
\end{equation}
on $\Prob(M)$ which satisfies the triangle inequality.

Though the initial references dealt specifically with the case $p=1$ \cite{KantorovichRubinstein58} \cite{Wasserstein69},
the whole family of distances are now called
{\em Kantorovich-Rubinstein} or {\em Wasserstein} metrics \cite{Dudley76}.
Apart from the interesting exception of the limiting case $d_\infty = \lim_{p \to \infty} d_p$
\cite{McCann06} \cite{ChampionDePascaleJuutinen08}, on a compact metric space $M$
all these metrics give rise to the same topology,
namely weak-$*$ convergence.  For non-compact $M$, the $d_p$
topologies differ from each other only in the number of moments of a sequence of measures which
are required to converge.  Moreover, $(\Prob(M),d_p)$ inherits geometric properties from $(M,d)$,
such as being a geodesic space.  Notions such as Ricci curvature in the underlying space $(M,d)$
can be characterized by the geodesic convexity first explored in \cite{McCann97}
of certain functionals on the larger space $\Prob(M)$ ---
such as Boltzmann's entropy. 
One direction of this equivalence was proved
by Cordero-Erausquin, McCann, and Schmuckensch\"ager
\cite{CorderoMcCannSchmuckenschlager01}
and Otto and Villani \cite{OttoVillani00} in projects which were initially independent
(see also \cite{Cordero-Erausquin99S} and \cite{CorderoMcCannSchmuckenschlager06}),
while the converse was established by von Renesse and Sturm
\cite{vonRenesseSturm05} confirming the formal arguments of \cite{OttoVillani00}.
This equivalence forms the basis of Lott-Villani and Sturm's
definition of lower bounds for Ricci curvature in the metric measure space setting,
without reference to any underlying Riemannian structure \cite{LottVillani09}
\cite{Sturm06ab}.  McCann and Topping used a similar idea to
characterize the Ricci flow \cite{McCannTopping10},
which led Lott \cite{Lott09} and Topping \cite{Topping09} to simpler proofs of Perelman's
celebrated monotonicity results \cite{Perelman02p}.  Despite the interest of these
recent developments, we shall not pursue them farther in these lectures,  apart from
sketching a transportation-based proof of the isoperimetric theorem whose ideas underpin
many such geometric connections.

\subsection{Brenier's theorem and convex gradients}

It turns out that Monge's cost $c(x,y)=|x-y|$ is among the hardest to deal with,
due to its lack of strict convexity.  For this cost, the minimizer of \eqref{Monge work}
is not generally unique, even on the line $M^\pm=\R$.
Existence of solutions is tricky to establish:   the first `proof', due to Sudakov \cite{Sudakov76},
relied on an unsubstantiated claim which turned out to be correct only in the
plane $M^\pm=\R^2$ \cite{BianchiniCavalletti10p};  higher dimensional
arguments were given increasing generality by Evans-Gangbo \cite{EvansGangbo99},  and then
Ambrosio \cite{Ambrosio03}, Caffarelli-Feldman-McCann \cite{CaffarelliFeldmanMcCann00},
and Trudinger-Wang \cite{TrudingerWang01} independently.
Simpler approaches were proposed by Champion-DePascale \cite{ChampionDePascale09p}
and Bianchini-Cavalletti \cite{BianchiniCavalletti10p} more recently.\\

The situation for the quadratic cost $c(x,y)=|x-y|^2$ is much simpler,
mirroring the relative simplicity of the Hilbert geometry of $L^2$ among
Banach spaces $L^p$ with $p \ge 1$. Brenier \cite{Brenier87} \cite{Brenier91}
(and others around the same time
\cite{PurserCullen87} \cite{SmithKnott87}
\cite{CuestaMatran89} \cite{RuschendorfRachev90} \cite{Cuesta-AlbertosTuero-Diaz93} \cite{AbdellaouiHeinich94})
showed that there is a unique  \cite{Cuesta-AlbertosTuero-Diaz93} \cite{AbdellaouiHeinich94} solution \cite{CuestaMatran89} ,
and characterized it as a convex gradient \cite{PurserCullen87} \cite{SmithKnott87} \cite{RuschendorfRachev90}.

\begin{thm}[A version of Brenier's theorem]
If $\mu^+ \ll dVol$ and $\mu^-$ are Borel probability measures on $M^\pm=\R^n$,
then there exists a convex function $u:\R^n \to \R \cup \{+\infty\}$
whose gradient $G=Du:\R^n \to\R^n$ pushes $\mu^+$ forward to $\mu^-$.
Apart from changes on a set of measure zero,
$G$ is the only map to arise in this way. Moreover, $G$ uniquely minimizes
Monge's problem \eqref{Monge work} for the cost $c(x,y)=|x-y|^2$.
\end{thm}

Remark:  In this generality the theorem was established by McCann \cite{McCann95},
where the assumption $\mu^+ \ll dVol$ was also relaxed.  A further relaxation by
Gangbo-McCann \cite{GangboMcCann96} is shown to be sharp in Gigli \cite{Gigli09p}.\\

\subsection{Fully-nonlinear degenerate-elliptic Monge-Amp\`ere type PDE}
\label{S:nonlinear PDE}
How do partial differential equations
(more specifically, fully nonlinear degenerate elliptic PDE) enter the picture?
Let's consider the constraint $G_\#\mu^+=\mu^-$, assuming moreover that
$\mu^{\pm}=f^{\pm}dVol^{\pm}$ on $\R^n$ or on Riemannian manifolds $M^\pm$.
Then if $\phi \in C(M^-)$ is a test function, it follows that
$$
\int_{M^+} \phi(G(x))f^+(x) dVol^+(x)=\int_{M^-}\phi(y)f^-(y) dVol^-(y).
$$
If $G$ was a diffeomorphism, we could combine the Jacobian factor
$d^n y = |\det DG(x)| d^n x$ from the change of variables $y=G(x)$ with arbitrariness
of $\phi\circ G$ to conclude $f^+(x)=|\det DG(x)|f^-(G(x))$ for all $x$.
We will actually see that this nonlinear equation holds $f^+$-a.e.\ as a consequence
of Theorem \ref{T:McCannPassWarren}.

In the case of Brenier's map $G(x)=Du(x)$, convexity of $u$ implies non-negativity
of the Jacobian $DG(x)=D^2u(x)\geq 0$. It
also guarantees almost everywhere differentiability of $G$ by Alexandrov's theorem
(or by Lebesgue's theorem in one dimension);  see Theorem \ref{T:McCannPassWarren}
for the sketch of a proof.
Thus $u$ solves the Monge-Amp\'ere equation \cite{Brenier91}
\begin{equation}\label{Monge-Ampere}
f^+(x)=\det(D^2u(x))f^-(Du(x))
\end{equation}
a.e.\  \cite{McCann97} subject to the condition $Du(x) \in M^-$ for $x \in M^+$.
This is known as the {\em 2nd boundary value problem} in the
partial differential equations literature.  We shall see that
linearization of this equation around a {\em convex} solution leads
to a (degenerate) elliptic operator \eqref{uniformly elliptic linearization},
whose ellipticity becomes uniform
if the solution $u$ is smooth and {\em strongly} convex, meaning positivity
of its Hessian is strict: $D^2 u(x)>0$.\\

\subsection{Applications}
The Monge-Kantorovich theory has found a wide variety of applications in pure
and applied mathematics.  On the pure side,  these include connections to
inequalities  \cite{McCann94} \cite{Trudinger94} \cite{McCann97}
\cite{Cordero-ErausquinNazaretVillani04} \cite{MaggiVillani08} \cite{FigalliMaggiPratelli10p},
geometry (including sectional \cite{Loeper09} \cite{KimMcCann07p}, Ricci
\cite{LottVillani09} \cite{Lott09} \cite{Sturm06ab} \cite{McCannTopping10}
and mean \cite{KimMcCannWarren09p} curvature),
nonlinear partial differential equations
\cite{Brenier87} \cite{Caffarelli92} \cite{Caffarelli96} \cite{Urbas97}
\cite{MaTrudingerWang05},
and dynamical systems
(weak KAM theory \cite{BernardBuffoni07};
nonlinear diffusions \cite{Otto01};
gradient flows \cite{AmbrosioGigliSavare05}).
On the applied side these
include applications to vision (image registration and morphing \cite{HakerZhuTannenbaumAngenent04}),
economics (equilibration of supply with demand \cite{Ekeland05} \cite{ChiapporiMcCannNesheim10},
structure of cities \cite{CarlierEkeland08},
maximization of profits \cite{RochetChone98} \cite{Carlier01} \cite{FigalliKimMcCann-econ}
or social welfare \cite{FigalliKimMcCann-econ}),
physics \cite{Dobrushin70} \cite{Tanaka73} \cite{McCann98} \cite{FigalliMaggi10p},
engineering
(optimal shape / material design \cite{BouchitteButtazzo01} \cite{BouchitteGangboSeppecher08},
reflector antenna design \cite{GlimmOliker03} \cite{Wang96} \cite{Wang04},
aerodynamic resistance \cite{Plakhov04b}), 
atmosphere and ocean dynamics (the semigeostrophic theory \cite{CullenPurser84} \cite{CullenPurser89} \cite{Cullen06}),
biology (irrigation \cite{BernotCasellesMorel09}, leaf growth \cite{Xia07}),
and statistics \cite{RachevRuschendorf98}.
See \cite{Villani03} \cite{Villani09} for further directions, references, and discussion.

\subsection{Euclidean isoperimetric inequality}

It was observed (independently by
McCann \cite{McCann94} \cite{McCann97} and Trudinger \cite{Trudinger94})
that a solution to the second boundary value problem for the Monge-Amp\`ere equation
\eqref{Monge-Ampere} yields a simple proof of the isoperimetric inequality (with its
sharp constant): for $M^+ \subset \R^n$
\begin{equation}\label{isoperimetry}
Vol(M^+) = Vol(B_1) \quad \Rightarrow \quad
\Hnm(\partial M^+) \geq \Hnm(\partial B_1).
\end{equation}
The following streamlined argument was perfected later;
it combines optimal maps with 
an earlier approach from Gromov's appendix to~\cite{MilmanSchechtman86}.

\begin{proof}
Take $f^+=\chi_{M^+}$ and $f^-=\chi_{B_1}$ to be uniformly distributed.
Brenier's theorem then gives a volume-preserving map $G=Du$ between $M^+$ and $B_1$:
$$
1=det^{1/n}(D^2u(x)).
$$
The expression on the right is the geometric mean of the eigenvalues of $D^2 u(x)$,
which are non-negative by convexity of $u$,  so
the arithmetic-geometric mean inequality yields
\begin{equation}\label{g/am}
1\leq (arithmetic \; mean \;of \; eigenvalues)=\frac{1}{n}\Delta u
\end{equation}
almost everywhere in $M^+$.  (The right hand side is the absolutely
continuous part of the distributional Laplacian;  convexity of $u$
allows it to be replaced by the full distributional Laplacian of $u$ without
spoiling the inequality.)  Integrating inequality \eqref{g/am} on $M^+$ yields
\begin{equation}\label{int by parts}
Vol(M^+)\leq \frac{1}{n}\int_{M^+}\Delta u \; d^nx = \frac{1}{n} \int_{\partial M^+} D u(x) \cdot \hat n_{M^+}(x) d\Hnm(x).
\end{equation}

Now, since $G=Du \in B_1$ whenever $x \in M^+$, we have $|D u |\leq 1$, thus
$Vol(M^+)=Vol(B_1)$ gives
\begin{equation}\label{comparison}
Vol(B_1)\leq \frac{1}{n}\int_{\partial M^+} 1\, d\Hnm = \frac{1}{n} \Hnm(\partial M^+).
\end{equation}
In the case special $M^+ = B_1$, Brenier's map coincides with the identity map
so equalities hold throughout \eqref{g/am}--\eqref{comparison},
yielding the desired conclusion \eqref{isoperimetry}!
\end{proof}

As the preceding proof shows,
one of the important uses of optimal transportation
in analysis and geometry is to encode non-local `shape' information
into a map which can be localized,  reducing global geometric inequalities to
algebraic inequalities under an integral.  For subsequent developments in this direction,
see works of Ambrosio, Cordero-Erausquin, Carrillo, Figalli, Gigli, Lott, Maggi, McCann, Nazaret,
Otto, Pratelli, von Renesse, Schmuckenschl\"ager, Sturm, Topping and Villani in
\cite{McCann98} \cite{OttoVillani00} \cite{CorderoMcCannSchmuckenschlager01}
\cite{CorderoMcCannSchmuckenschlager06} \cite{Cordero-ErausquinNazaretVillani04}
\cite{AmbrosioGigliSavare05}  \cite{CarrilloMcCannVillani06} \cite{vonRenesseSturm05}
\cite{Sturm06ab} \cite{LottVillani09} \cite{Lott09}
\cite{McCannTopping10} \cite{FigalliMaggiPratelli10p}.


\subsection{Kantorovich's reformulation of Monge's problem}

Now let us turn to the proof of Brenier's theorem and the ideas it involves.
A significant breakthrough was made by Kantorovich \cite{Kantorovich42} \cite{Kantorovich48},
who \emph{relaxed} our optimization problem (the Monge problem), by dropping the requirement
that all the ore from a given mine goes to a single factory. In other words:\\

Replace $G:M^+\to M^-$ by a measure $0 \leq \gamma$ on $M^+\times M^-$ whose marginals are
$\mu^+$ and $\mu^-$, respectively, and among such measures choose $\gamma$ to minimize the
functional

$$\cost(\gamma)=\int_{M^+\times M^-}c(x,y)d\gamma(x,y).$$

Such a joint measure $\gamma$ is also known as a ``transport plan'' (in analogy with``transport map''). This is better than Monge's original formulation for at least two reasons:\\

1) The functional to be minimized now depends linearly on $\gamma$.\\

2) The set $\Gamma(\mu^+,\mu^-)$ of admissible competitors $\gamma$ is a convex subset of a suitable
 Banach space: namely,  the dual space to continuous functions $(C(M^+ \times M^-),\|\cdot\|_\infty)$
 (which decay to zero at infinity in case the compactness of $M^\pm$ is merely local). \\

In this context, well-known results in functional analysis guarantee existence of a
minimizer $\gamma$ under rather general hypotheses on $c$ and $\mu^\pm$.
Our primary task will be to understand when the solution will be unique, and to characterize it.
At least one minimizer will be an extreme point of the convex set $\Gamma(\mu^+,\mu^-)$,
but its uniqueness remains an issue.
Necessary and sufficient conditions will come from the duality theory
of (infinite dimensional) linear programming \cite{AndersonNash87} \cite{Kellerer84}.

\section{Existence, uniqueness, and characterization of optimal maps}

Let's get back to the Kantorovich problem:

\begin{equation}\label{Kantorovich}
W_c(\mu^+,\mu^-)
:= \min \limits_{\gamma \in \Gamma(\mu^+,\mu^-)}\int_{M^+\times M^-}c(x,y)d\gamma(x,y)
= \min \limits_{\gamma \in \Gamma(\mu^+,\mu^-)} \cost(\gamma)
\end{equation}

The basic geometric object of interest to us will be the support
$\spt \gamma := S$ of a competitor $\gamma$, namely the smallest closed subset
$S \subset M^+ \times M^-$ carrying the full mass of $\gamma$.\\

What are some of the competing candidates for the minimizer?\\

eg.1) Product measure: $\mu^+ \otimes \mu^- \in \Gamma(\mu^+,\mu^-)$, for which
$\spt \gamma = \spt \mu^+ \times \spt \mu^-$. \\

eg.2) Monge measure: if $G: M^+ \to M^-$ with $G_\#\mu^+=\mu^-$ then
$id \times G: M^+ \longrightarrow M^+ \times M^-$ and
$\gamma= (id\times G)_\#\mu^+ \in \Gamma(\mu^+,\mu^-)$ has $\cost(\gamma)=\cost(G)$.\\

The second example shows in what sense Kantorovich's formulation is a relaxation of Monge's problem,
and why \eqref{Monge work} must be at least as big as \eqref{Kantorovich}.  In this example
$\spt \gamma$ will be the (closure of the) graph of $G:M^+\longrightarrow M^-$,  which suggests how
Monge's map $G$ might in principle be reconstructed from a minimizing Kantorovich measure $\gamma$.
Before attempting this,  let us recall a notion which characterizes optimality
in the Kantorovich problem.

\begin{DEF}[$c$-cyclically monotone sets]
$S \subset M^+\times M^-$ is {\em $c$-cyclically monotone} if and only if all
$k \in \mathbf{N}$ and $(x_1,y_1),...,(x_k,y_k) \in S$ satisfy
\begin{equation}\label{c-cyclically monotone}
\sum \limits_{i=1}^kc(x_i,y_i) \leq \sum \limits_{i=1}^k c(x_i,y_{\sigma(i)})
\end{equation}
for each permutation $\sigma$ of $k$ letters.
\end{DEF}

The following result was deduced by Smith and Knott \cite{SmithKnott92}
from a theorem of R\"uschendorf \cite{Ruschendorf90}. A more direct proof was given
by Gangbo and McCann \cite{GangboMcCann96};  its converse is true as well.

\begin{thm}[Smith and Knott '92]\label{T:SmithKnott}
If $c \in C(M^+ \times M^-)$, then optimality of $\gamma \in \Gamma(\mu^+,\mu^-)$ implies
$\spt \gamma$ is a $c$-cyclically monotone set.
\end{thm}

The idea of the proof in \cite{GangboMcCann96} is that if $spt\;\gamma$ is not cyclically monotone,
then setting $o_i = (x_i,y_i)$ and $z_i=(x_i,y_{\sigma(i)})$ we could with some care
define a perturbation
$$
\gamma_\epsilon=\gamma + \epsilon \mbox{(near the z's)}-\epsilon \mbox{(near the o's)}
$$
in $\Gamma(\mu^+,\mu^-)$ of $\gamma$ for which $\cost(\gamma_\epsilon) < \cost(\gamma_0)$,
thus precluding the optimality of $\gamma$.

e.g. If $c(x,y)=-x\cdot y$ or $c(x,y)=\frac{1}{2}|x-y|^2$ then \eqref{c-cyclically monotone} becomes
$\sum \limits_{i=1}^k \langle y_i,x_{i}-x_{i-1}\rangle \ge 0$ with the convention $x_{0}:=x_k$.
This is simply called {\em cyclical monotonicity}, and can be viewed as a discretization of
$$
\oint \bar{y}(x)\cdot dx \ge 0,
$$
a necessary and sufficient condition for the vector field $\bar{y}(x)$ to be conservative,
meaning $y = Du(x)$. This heuristic underlies a theorem of Rockafellar \cite{Rockafellar66}:

\begin{thm}[Rockafellar '66]\label{T:Rockafellar}
The set $S \subset \R^n \times \R^n$ is cyclically monotone if and only if there exists a
convex function $u: \R^n \to \R\cup \{\infty\}$ such that $S \subset \partial u$ where
\begin{equation}\label{subdifferential}
\partial u := \{ (x,y) \in \R^n\times\R^n \mid
u(z) \ge  u(x) + \langle z -x,y \rangle +o(|z-x|) \quad \forall z \in \R^n\}.
\end{equation}
\end{thm}
The {\em subdifferential} $\partial u$ defined by \eqref{subdifferential}
consists of the set of $(point,slope)$ pairs for which $y$ is the slope of a
hyperplane supporting the graph of $u$ at $(x,u(x))$.

\begin{rem}[Special case ({\em monotonicity})]
Note that when $c(x,y)=-x\cdot y$ and $k=2$,
\eqref{c-cyclically monotone} implies for all
$(x_1,y_1), (x_2,y_2) \in S$ that 
\begin{equation}\label{monotone}
\langle \Delta x, \Delta y \rangle := \langle x_2-x_1,y_2-y_1 \rangle\geq 0.
\end{equation}
This condition implies that $y_2$ is constrained to lie in a halfspace with $y_1$ on its
boundary and $\Delta x$ as its inward normal.  Should $y=Du(x)$ already be known to be conservative,
the monotonicity inequalities \eqref{monotone} alone become equivalent to convexity of $u$.
\end{rem}

\subsection{Linear programming duality}
An even more useful perspective on these linear programming problems
is given by the the duality theorem discovered by Kantorovich \cite{Kantorovich42} and
Koopmans \cite{Koopmans49}
--- for which they later shared the Nobel Memorial Prize in economics. It states
that our minimization problem is equivalent to a maximization problem
\begin{equation}\label{duality}
\min\limits_{\gamma \in \Gamma(\mu^+,\mu^-)}\int_{M^+\times M^-} c \, d\gamma =
\sup \limits_{(-u,-v)\in Lip_c} -\int_{M^+}u(x) \, d\mu^+(x)-\int_{M^-}v(y) \, d\mu^-(y).
\end{equation}
Here
$$
Lip_c = \{ (u^+,u^-)\; with \; u^\pm\in L^1(M^\pm,d\mu^\pm)\;|\; c(x,y)\geq u^+(x)+u^-(y)\;\forall (x,y) \in M^+\times M^-\}.
$$

One of the two inequalities ($\ge$) in \eqref{duality} follows at once from the definition of
$-(u,v) \in Lip_c$ by integrating
$$
c(x,y) \ge - u(x) - v(y)
$$
against $\gamma \in \Gamma(\mu^+,\mu^-)$.
The magic of duality is that {\em equality} holds in \eqref{duality}.

\subsection{Game theory}
Some intuition
for why this magic works can be gleaned from the the theory of (two-player, zero-sum) games.
In that context, Player 1 chooses strategy $x \in X$,
Player 2 chooses strategy $y \in Y$, and the outcome is that
Player 1 pays $P(x,y)$ to Player 2.  The payoff function $P \in C(X \times Y)$
is predetermined and known in advance to both players;
P1 wants to minimize the resulting payment and P2 wants to maximize it.

Now, what if one the players declares his or her strategy ($x$ or $y$) to the other player
in advance?
If P1 declares first, the outcome is better for P2,  who has a chance to optimize his response
to the announced strategy $x$, and conversely.  This implies that
\begin{equation}\label{min max}
 \inf \limits_{x \in X} \sup \limits_{y \in Y} P(x,y)
\ge \sup \limits_{y\in Y}\inf \limits_{x \in X} P(x,y);
\end{equation}
(Player 1 declares first vs player 2 declares first.)
Von Neumann \cite{vonNeumann28}
identified structural conditions on the payoff function to have a
saddle point \eqref{saddle}, in which case equality holds in \eqref{duality};
see also Kakutani's reference to \cite{vonNeumann37} in \cite{Kakutani41}.

\begin{thm}[convex/concave min-max]
If $X \subset \R^m$ and $Y \subset \R^n$ are compact and convex,
then equality holds in \eqref{min max} provided for each $(x_0,y_0) \in X \times Y $ both
functions $x \in X \longmapsto P(x,y_0)$ and $y \in Y \longmapsto -P(x_0,y)$ are convex.
(In fact, convexity of all sublevel sets of both functions is enough.)
\end{thm}

\begin{proof}
Let
$$
 x_b(y) \in \arg\min_{x \in X} P(x,y) \qquad y_b(x) \in \arg\max_{y \in Y} P(x,y)
$$
denote the best responses of P1 and P2 to each other's strategies $y$ and $x$.
Note $x_b$ and $y_b$ are continuous if the convexity and concavity assumed of
the payoff function are both strict.  In that case,  Brouwer's theorem
asserts the the function $y_b \circ x_b: Y \to Y$ has a fixed point $y_0$.
Setting $x_0 = x_b(y_0)$,  since $y_0=y_b(x_0)$ we have found a saddle point
\begin{equation}\label{saddle}
\inf_{x \in X} \max_{y \in Y} P(x,y) \le
\max_{y \in Y} P(x_0,y) = P(x_0,y_0) =
\min_{x \in X} P(x, y_0) \le
\sup_{y \in Y} \min_{x \in X} P(x, y)
\end{equation}
of the payoff function, which proves equality holds in \eqref{min max}.
If the convexity and concavity of the payoff function are not strict,
apply the theorem just proved to the
perturbed payoff $P_\epsilon(x,y) = P(x,y) + \epsilon (|x|^2 - |y|^2)$
and take the limit $\epsilon \to 0$.
\end{proof}

e.g. Expected payoff \cite{vonNeumann28}:  von Neumann's original example of a function
to which the theorem and its conclusion applies
is the expected payoff $P(x,y) = \sum_{i=1}^m \sum_{j=1}^n p_{ij} x_i y_j$ of
{\em mixed} or {\em randomized} strategies $x$ and $y$
for a game in which P1 and P2 each have only finitely
many pure strategies, and the payoff corresponding to strategy
$1 \le i \le m$ and $1\le j\le n$  is $p_{ij}$.  In this case
$X = \{ x \in [0,1]^m \mid \sum x_i =1\}$ and $Y = \{ y \in [0,1]^n \mid \sum y_j = 1\}$
are standard simplices of the appropriate dimension, whose vertices correspond to
the pure strategies.

\subsection{Relevance to optimal transport: Kantorovich-Koopmans duality}
Infinite dimensional versions of von-Neumann's theorem can also be formulated
where $X$ and $Y$ lie in Banach spaces;  they are proved using Schauder's fixed point theorem
instead of Brouwer's.  A payoff function germane to optimal transportation is defined on the
strategy spaces $X = \{ 0 \le \gamma \; on \; M^+\times M^-\}$ and
$Y = \{ (u,v) \in L^1(M^+,\mu^+) \oplus L^1(M^-,\mu^-)\}$ by
$$
P(\gamma, (u,v))= \int_{M^+\times M^-}\left (c(x,y)+u(x)+v(y) \right ) d\gamma(x,y)-\int_{M^+}ud\mu^+ -\int_{M^-}vd\mu^-.
$$
Note the bilinearity of $P$ on $X \times Y$.  Since
$$
\inf_{\gamma \in X} P(\gamma,(u,v)) = \left \{ \begin{array}{ll}
-\infty & \mbox{ unless } (-u,-v) \in Lip_c,\\
0 -\int u d\mu^+ -\int vd\mu^- & \mbox{ otherwise,}
\end{array} \right.$$
the Kantorovich-Koopmans dual problem is recovered from the version of the game in which
P2 is compelled to declare his strategy first:
\begin{equation}\label{dual}
\sup_{(u,v)\in Y} \inf \limits_{\gamma \in X} P(\gamma,(u,v)) = \sup \limits_{(-u,-v) \in Lip_c} \int (-u) d\mu^+ +\int (-v) d\mu^-.
\end{equation}
On the other hand, rewriting
$$
P(\gamma,(u,v)) = \int_{M^+ \times M^-}cd\gamma+\int u (d\gamma - d\mu^+)+\int v(y)(d\gamma(x,y)-d\mu^-(y))
$$
we see
\begin{equation}\label{primal}
\sup \limits_{(u,v)\in Y} P(\gamma,(u,v)) = \left \{ \begin{array}{ll}
+\infty & \mbox{ unless } \gamma \in \Gamma(\mu^+,\mu^-)\\
\int_{M^+ \times M^-} c \, d\gamma & \mbox{ if } \gamma \in \Gamma(\mu^+,\mu^-).
\end{array} \right.
\end{equation}
Thus the primal transportation problem of Kantorovich and Koopmans
$$\inf \limits_{\gamma \in X}\sup \limits_{(u,v)\in Y} P(\gamma,(u,v))=\inf \limits_{\gamma \in \Gamma(\mu^+,\mu^-)} \int_{M^+ \times M^- } c \,d\gamma
$$
corresponds to the version of the game in which P1 declares his strategy first.
The equality between \eqref{dual} and \eqref{primal} asserted by an appropriate generalization
of von Neumann's theorem implies the duality \eqref{duality}:
$$
\min \limits_{\gamma \in \Gamma(\mu^+,\mu^-)}\int_{M^+ \times M^-}cd\gamma
= \sup \limits_{(-u,-v) \in Lip_c} \int_{M^+}(-u)d\mu^+ + \int_{M^-}(-v)d\mu^-.$$

\subsection{Characterizing optimality by duality}

The following theorem can be deduced as an immediate corollary of this duality.
We may think of the potentials $u$ and $v$ as being Lagrange multipliers enforcing
the constraints on the marginals of $\gamma$;  in the economics literature they are
interpreted as shadow prices which reflect the geographic variation in scarcity or abundance
of supply and demand.  The geography is encoded in the choice of cost.


\begin{thm}[Necessary and sufficient conditions for optimality]\label{T:zero set}
The existence of $-(u,v) \in Lip_c$ such that $\gamma$
vanishes outside the zero set of the non-negative function
$k(x,y) = c(x,y)+u(x)+v(y) \geq 0$ on $M^+ \times M^-$
is necessary and sufficient for the optimality of $\gamma \in \Gamma(\mu^+,\nu^-)$
with respect to $c \in C(M^+ \times M^-)$.
\end{thm}

\begin{cor}[First and second order conditions on potentials]\label{C:foc}
Optimality of $\gamma$ and $(u,v)$ implies $Dk = 0$ and $D^2 k \ge 0$ at any point
$(x,y) \in \spt \gamma$ where these derivatives exist.
In particular, $D_x[c(x,y)+u(x)+v(y)]=0$ and $D^2_x[c(x,y)+u(x)+v(y)]\geq 0$ holds $\gamma$-a.e.,
and likewise for $y$-derivatives.
\end{cor}

e.g. Consider the special case of the bilinear cost: $c(x,y)=-x\cdot y$.
Here the first and second order conditions of the corollary become
$$
y = Du(x) \qquad {\rm and} \qquad D^2u(x) \ge 0,
$$
suggesting $y$ is the graph of the gradient of a convex function.
In this case, convexity of $u$ guarantees $D^2u$ is defined a.e. with respect to Lebesgue measure,
by Alexandrov's theorem.

\subsection{Existence of optimal maps and uniqueness of optimal measures}

More generally, we claim $u$ inherits Lipschitz and semiconvexity bounds
\eqref{locally Lipschitz}--\eqref{semiconvex} from $c(x,y)$, which guarantee the
existence of $x$-derivatives in the preceding corollary --- at least Lebesgue almost everywhere.
This motivates the following theorem of Gangbo \cite{Gangbo95} and Levin \cite{Levin99};
variations appeared independently in Caffarelli \cite{Caffarelli96},
Gangbo and McCann \cite{GangboMcCann95} \cite{GangboMcCann96}, and R\"uschendorf
\cite{Ruschendorf95} \cite{Ruschendorf96} at around the same time,
and subsequently in  
\cite{MaTrudingerWang05}.

\begin{DEF}[Twist conditions]
A function $c \in C(M^+ \times M^-)$ differentiable with respect to
$x \in M^+$ is said to be {\em twisted} if
\begin{equation}\label{Aone}
\Aone^+ \quad 
\forall\ x_0 \in M^+, {\rm\ the\ map}\
    y \in M^- \longmapsto D_x c(x_0,y) \in T^*_x M^+ \mbox{\rm\ is\ one-to-one}.
\end{equation}
For $(x,p) \in T^*M^+$ denote the unique $y \in M^-$ solving $D_x c(x,y) + p =0 $ by
$y=Y(x,p)$ when it exists.  When the same condition holds for the cost
$\tilde c(y,x) := c(x,y)$, we denote it by $\Aone^-$.  When both $c$ and $\tilde c$
satisfy $\Aone^+$,  we say the cost is {\em bi-twisted}, and denote this by
\Aone.
\end{DEF}

\begin{thm}[Existence of Monge solutions; uniqueness of Kantorovich solutions]
\label{T:GangboLevin}
    Fix Polish probability spaces $(M^\pm,\mu^\pm)$
    and assume $M^+$ is a $n$-dimensional manifold and $d\mu^+ \ll d^n x$ is absolutely
    continuous (in coordinates).  Let $c \in C(M^+ \times M^-)$ differentiable with
    respect to $x\in M^+$ satisfy the twist condition \eqref{Aone} and assume
    $D_x c(x,y)$ 
    is bounded locally in $x \in M^+$ uniformly in $y \in M^-$.
Then, there exists a locally Lipschitz (moreover, $c$-convex, as in Definition \ref{c-convex})
function $u: M^+ \to \R$ 
such that

    a) $G(x):=Y(x,Du(x))$ pushes $\mu^+$ forward to $\mu^-$;\\

    b) this map is unique, and uniquely solves Monge's minimization problem \eqref{Monge work};\\

    c) Kantorovich's minimization \eqref{Kantorovich} has a unique solution $\gamma$;\\

    d) $\gamma = (id \times G)_\#\mu^+$.

\end{thm}

\begin{DEF}[c-convex]\label{c-convex}
A function $u:M^+ \longrightarrow \R \cup \{+\infty\}$ (not identically
infinite) is {\em c-convex} if and only if $u=(u^{\tc})^c$,  where
\begin{equation}\label{c-transform}
u^{\tc}(y)= \sup \limits_{x \in M^+} -c(x,y)-u(x)
\qquad {\rm and} \qquad v^c(x)= \sup \limits_{y \in M^-} -c(x,y)-v(y).
\end{equation}
\end{DEF}

\begin{rem}[Legendre-Fenchel transform and convex dual functions]\label{Legendre-Fenchel}
When $c(x,y) = -\langle x,y \rangle$,  then $u^\tc(y)$ is manifestly convex: it is the
Legendre-Fenchel transform or convex dual function of $u(x)$.  In this case,
$(u^\tc)^c$ is well-known to yield the lower semicontinuous convex hull of the graph of $u$,
so that $u=(u^\tc)^c$ holds if and only if $u$ is already lower semicontinuous and convex.
More generally,  we interpret the condition $u=u^{\tc c}$ as being the correct adaptation of the
notion of convexity to the geometry of the cost function $c$.
\end{rem}

\begin{proof}[Sketch of proof of Theorem \ref{T:GangboLevin}]
The key idea of the proof is to establish existence of a
maximizer $-(u,v) \in Lip_c$ of \eqref{duality}
with the additional property that $(u,v)=(v^c,u^\tc)$.
Differentiability of $u=u^{\tc c}$ on a set $\dom Du$ of full
$d\mu^+ \ll d^n x$ measure then follows from Rademacher's theorem and Lemma \ref{L:Lipschitz}.
The map $G(x) := Y(x,Du(x))$ is well-defined on $\dom Du$ by the twist condition
(assuming the supremum \eqref{c-transform} defining $(u^\tc)^c(x)$ is attained).
Corollary \ref{C:foc} shows any minimizer $\gamma$ vanishes outside the graph of this map,
and it then follows easily that $\gamma = (id \times G)_\# \mu^+$ and hence $\gamma$ is uniquely
determined by $u$ \cite{AhmadKimMcCann09p}.
Conversely any other $c$-convex $\tilde u$ for which
$\tilde G(x) = Y(x,D \tilde u(x))$ pushes $\mu^+$ forward to $\mu^-$ can be shown to
maximize the dual problem by checking that $\tilde \gamma = (id \times \tilde G)_\# \mu^+$
vanishes outside the support of $\tilde G$.  Thus $\tilde \gamma = \gamma$ and
$\tilde G = G$ holds $\mu^+$-a.e.

To extract the desired $-(u,v) \in Lip_c$ from a maximizing sequence $-(u_k,v_k)$
requires some compactness. (This would come from the convexity of $u_k$ and $v_k$
in case $c(x,y) = - x \cdot y$ via the Blaschke selection theorem.)
Observe $-(u,v)\in Lip_c$ implies
$$u(x) \ge \sup \limits_{y \in M^-} -c(x,y)-v(y)=:v^c(x).$$
Moreover, $-(v^c,v) \in Lip_c$ and $-(v^c) \ge -u$ can only increase the value of the objective
functional relative to $-(u,v)$.  Thus $-(v^c,v)$ is a better candidate for a maximizer
than $-(u,v)$. Repeating the process shows $-(v^c,v^{c \tc})$
and $-(v^{c \tc c}, v^{c \tc}) \in Lip_c$ are better still,  since $(v^c)^\tc \le v$ and
$(v^{c \tc})^c \le v^c$ by the same logic.
On the other hand, starting from $v^{c \tc} \le v$,
the negative coefficient in definition \eqref{c-transform}
implies the opposite inequality $(v^{c \tc})^c \ge v^c$.
Thus $v^{c \tc c} = v^c$ quite generally.  (This is precisely analogous to the fact that
the second Legendre transform $u^{**}$ does not change a function $u=v^*$ which is already convex and
lower semicontinuous;  see Remark \ref{Legendre-Fenchel}.)

Replacing a maximizing sequence $-(u_k,v_k)$ with $-(v_k^c,v_k^{c \tc})$ therefore yields a new
maximizing sequence at least as good which moreover consists of $c$-convex functions.
Lemma \ref{L:Lipschitz} shows this new family is locally equi-Lipschitz,  hence
we only need local boundedness for the Arzela-Ascoli theorem to yield a limiting maximizer $-(u,v)$,
which will in fact be $c$-convex,  though we can also replace it by $-(v^c, v^{c \tc})$ just
to be sure.  Local boundedness also follows from Lemma \ref{L:Lipschitz}, after fixing
$x_0 \in \spt \mu^+ \subset M^+$ and
replacing $-(u_k,v_k)$ by $-(u_k - \lambda_k, v_k + \lambda_k)$ with $\lambda_k = u_k(x_0)$.
This replacement does not change the value of the objective functional \eqref{dual},
yet ensures that $u(x_0)=0$.
\end{proof}

\section{Methods for obtaining regularity of optimal mappings}

Given mines and factories $(M^\pm,\mu^\pm)$ and a cost function $c \in C(M^+\times M^-)$,
in the preceding section we found conditions which guarantee the existence and uniqueness
of a map $G(x)=Y(x,Du(x))$ such that $G_\#\mu^+=\mu^-$ with $u=u^{\tc c}$, ie.\ $c$-convex.
Under the same conditions,  the map $G$ is the unique minimizer of Monge's problem
\eqref{Monge work}.
The space $M^+$ was assumed to be an $n$-dimensional manifold, and the following twist
hypothesis $\Aone^+$, equivalent to \eqref{Aone}, was crucial to specifying $Y(x,\cdot)$:
\begin{equation}\label{topological twist}
\forall\ y_1 \ne y_2 \in M^- {\rm\ assume}\ x\in M^+ \longmapsto c(x,y_1)-c(x,y_2) {\rm\ has\ no\ critical\ points.}
\end{equation}
Notice, however, that \eqref{topological twist} cannot be satisfied by any cost function which is
{\em differentiable}  throughout a compact manifold $M^+$.
In case $M^+ = \sphere^n$,  Monge solutions do not generally exist \cite{GangboMcCann00},
but criteria are given in \cite{ChiapporiMcCannNesheim10} \cite{AhmadKimMcCann09p} which
guarantee uniqueness of the Kantorovich minimizer.  On the other hand, it is an interesting
open problem to find a criterion on $c \in C^1(M^+ \times M^-)$ which guarantees uniqueness
of Kantorovich solutions for all $\mu^\pm \in L^1(M^\pm)$ in more complicated topologies,
such as the torus $M^\pm = \torus^n$ for example.
Here differentiability of the cost function is crucial;
for costs such as Riemannian distance squared,  the desired uniqueness is
known \cite{Cordero-Erausquin99T} \cite{McCann01}, but the cost fails to be differentiable
at the cut locus.

\subsection{Rectifiability: differentiability almost everywhere}
The current section is devoted to reviewing methods for exploring the smoothness properties
of the optimal map $G$ found above, or equivalently of its $c$-convex potential $u$.
The following lemma shows that all $c$-convex functions inherit Lipschitz and semiconvexity properties directly
from the cost function $c$; it has already been exploited to prove Theorem \ref{T:GangboLevin}.

\begin{lem}[Inherent regularity of $c$-convex functions]\label{L:Lipschitz}
If $u=u^{\tc c}$ and $c(\cdot,y) \in C^k_{loc}(M^+)$ for each $y \in M^-$,
then $k=1$ implies \eqref{locally Lipschitz} and $k=2$ implies \eqref{semiconvex}:
\begin{align}\label{locally Lipschitz}
|Du(x)| &\leq \sup \limits_{y\in M^-} |D_x c(x,y)|  &{\rm (local\ Lipschitz\ regularity);} \\
\label{semiconvex}
D^2u(x) &\geq \inf \limits_{y \in M^-} -D^2_{xx}c(x,y) &{\rm (semiconvexity).}
\end{align}
\end{lem}

Similarly,  $c$-cyclically monotone sets $S \subset M^+ \times M^-$ turn out to
be contained in Lipschitz submanifolds of dimension $n = dim M^\pm$ when the cost
function is non-degenerate \eqref{Atwo}.  The following recent theorem of
McCann-Pass-Warren \cite{McCannPassWarren10p} combines with Rademacher's theorem
--- which asserts the differentiability Lebesgue a.e.\ of Lipschitz functions ---
to give a simple tool for establishing that $f^+(x)=|\det(DG(x))|f^-(G(x))$ holds $f^+$-a.e.\\

\begin{thm}[Rectifiability of optimal transport \cite{McCannPassWarren10p}]
\label{T:McCannPassWarren}
Assume $M^\pm$ are $n$-dimensional manifolds, at least in a neighbourhood $U$ of
$(x_0,y_0) \in M^+ \times M^-$, where $c \in C^2(U)$ and
\begin{equation}\label{Atwo}
\Atwo\;\; det D^2_{x^iy^j}c(x_0,y_0) \neq 0.
\end{equation}
If $S \subset M^+\times M^-$ is $c$-cyclically monotone,
then $S \cap V$ lies in an $n$-dimensional Lipschitz submanifold,  for some neighbourhood
$V \subset U$ of $(x_0,y_0)$.
\end{thm}

In view of Theorem \ref{T:SmithKnott}, this conclusion applies either to the graph
$S = Graph(G)$ of any optimal map \eqref{Monge work} or the support $S=\spt(\gamma)$
of any optimal measure \eqref{Kantorovich} in the transportation problem.

\begin{center}
\includegraphics[scale=0.30]{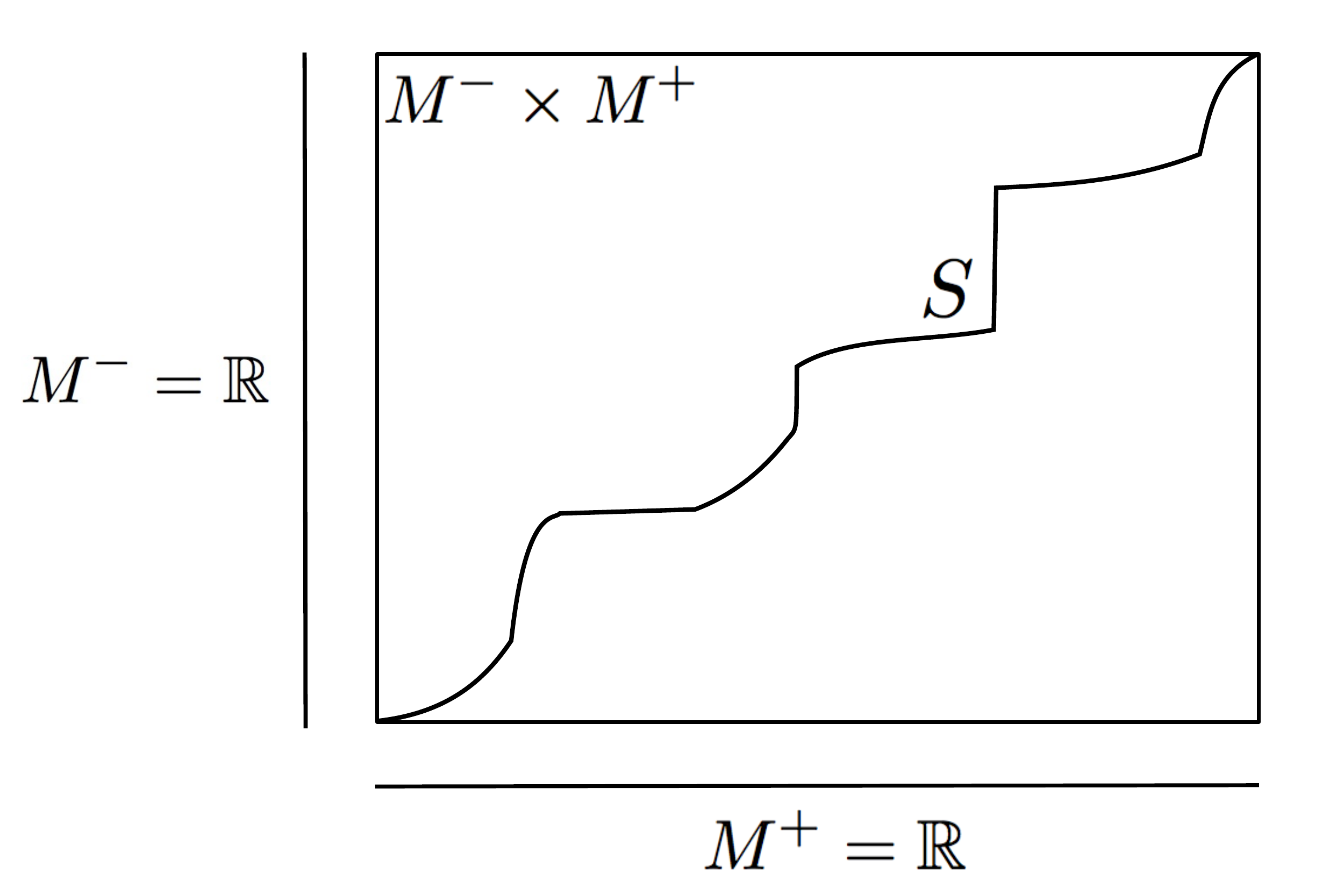}
\end{center}
\centerline{Figure 1: Optimizers have locally monotone support $S$ in the plane.}

\smallskip\noindent
Motivation: on the line $M^\pm=\R$, without further assumptions on $c$ or $\mu^\pm$, a
transport map may not exist nor be monotone, yet the theorem above says that even so all
pieces of $spt\;(\gamma)$ lie along along Lipschitz arcs in the plane. These curves
will actually be locally monotone --- non-decreasing or non-increasing depending on the
sign of $D^2_{xy} c(x_0, y_0)$; see Figure 1.

\begin{proof}[Idea of proof]

Introduce the notation $b=-c$.  In case $c(x,y)=-x\cdot y$ on $M^\pm =\R^n$,
monotonicity asserts for all $(x_0,y_0), (x_1,y_1) \in S$ that
$\Delta x = x_1-x_0$ and $\Delta y = y_1-y_0$ satisfy
$$
0\leq \left  < \Delta x, \Delta y \right >=\left < \frac{\Delta z- \Delta w}{\sqrt{2}}, \frac{\Delta z + \Delta w}{\sqrt{2}}\right >
$$
where
\begin{equation}\label{Minty rotation}
(z,w) := (\frac{x+y}{\sqrt{2}}, \frac{x-y}{\sqrt{2}}).
\end{equation}
This implies that
$|\Delta w|^2\leq |\Delta z|^2$
meaning $w=w(z)$ has Lipschitz constant 1 as a graph over $z \in \R^n$.
Equivalently,  $S$ has Lipschitz constant 1 as a graph over the diagonal in $M^+ \times M^-$.
This special case was established by Alberti and Ambrosio \cite{AlbertiAmbrosio99},
using an argument of Minty \cite{Minty62}.

For more general costs $b=-c$ and any $\epsilon>0$,
the non-degeneracy \eqref{Atwo} implies the existence of
new coordinates $\tilde{y}=\tilde{y}(y)$ on $M^-$ in a neighbourhood of $y_0$
such that $\tilde{b}(x,\tilde{y}(y))=b(x,y)$ satisfies
$|D^2_{x\tilde{y}}\tilde{b}(x,\tilde{y})-I|<\epsilon$
in a neighborhood $\tilde V$ of $(x_0,\tilde y_0)$ which is convex in coordinates.

Now,
$\langle \Delta x, \Delta \tilde y \rangle \geq -\epsilon |\Delta x| \Delta \tilde y|$ follows from
\begin{equation}
0  \le \tilde{b}(x_0,\tilde y_0) + \tilde{b}(x_1,\tilde y_1) - \tilde b (x_0,\tilde y_1) -\tilde b(x_1,\tilde y_0)
 =D^2_{x\tilde y} \tilde b(x^*,\tilde y^*)(x_1-x_0)(\tilde y_1-\tilde y_0)
\end{equation}
and the change of variables analogous to \eqref{Minty rotation} yields
$$
|\Delta z|^2-|\Delta w|^2 \geq -\epsilon |\Delta w-\Delta z||\Delta w+\Delta z|
\geq -\epsilon \left ( |\Delta w|^2+|\Delta z|^2 \right ).
$$
Thus $(1+\epsilon) |\Delta z|^2 \geq  (1-\epsilon)|\Delta w|^2$, which shows
$w=w(z)$ is again a Lipschitz function of $z \in \R^n$ in the chosen coordinates.
\end{proof}

\subsection{From regularity a.e.\ to regularity everywhere}
The regularity results discussed so far --- Lipschitz continuity of the potential $u$,
and of $Graph(G) \subset M^+ \times M^-$ rather than of the map $G(x) = Y(x,Du(x))$ itself ---
required no hypotheses on the probability measures $\mu^-=G_\# \mu^+$.  To address the continuity,
differentiability, and higher regularity everywhere for the map $G:M^+ \longrightarrow M^-$
is a much more
delicate issue which certainly requires further hypotheses on the data $\mu^\pm$ and $c$.
For example,  if $\spt \mu^-$ is connected but $\spt \mu^+$ is not,  then $G$ cannot be
continuous.  The same reasoning makes it clear that ellipticity of the Monge-Amp\`ere equation
\eqref{Monge-Ampere} cannot be non-degenerate for all convex solutions;
regularity must propagate from boundary
conditions since the {\em purely local} effect of the equation is insufficent to conclude
$u \in C^1_{loc}$.
It is often easier to work with the scalar potential $u$ rather than the mapping $G$;
we shall see this reduces the problem to a question in the theory of second-order, fully-nonlinear,
degenerate-elliptic partial differential equations \eqref{MA type} generalizing the
Monge-Amp\`ere equation.  However, this question was answered
first in the special case $c(x,y) = - x \cdot y$ corresponding to the case
\eqref{Monge-Ampere} by Delano\"e in the plane $n=2$
\cite{Delanoe91},  and by Caffarelli and Urbas in higher dimensions $M^\pm = \R^n$
\cite{Caffarelli92} \cite{Caffarelli92b} \cite{Caffarelli96b} \cite{Urbas97}.

\begin{rem}
Note for $c \in C^{k+1}(M^+ \times M^-)$ that
$u \in C^{k+1}$ implies $G \in C^{k}$ by the following remark.
Whereas the twist condition $\Aone^+$\ asserts that the definition $Y(x,p)$ by
$D_x c(x,Y(x,p)) + p = 0$ is unambiguous,  non-degeneracy \Atwo\ allows the implicit
function theorem to be applied to conclude $C^k$ smoothness of  $Y(x,p)$ (where defined).
\end{rem}


\subsection{Regularity methods for the Monge-Amp\`ere equation; renormalization}

There are several methods for obtaining regularity results for convex solutions of
the Monge-Amp\`ere equation.  The first to be discussed here is the {\em continuity method}, used
for example by Delano\"e '91 and Urbas '97.  This approach requires relatively strong assumptions
on the smoothness of the measures $d\mu^\pm = f^\pm dVol$ and on the convexity and smoothness of their domains
$M^\pm \subset \R^n$.  When it applies, it yields global regularity of the resulting potential
up to the boundary of $\partial M^+$ from the same fixed point argument which shows a solution
exists.  The second method is the {\em renormalization method} pioneered by Caffarelli '92-'96,
which starts from the unique (weak) solution to the 2nd boundary value problem and uses affine
invariance of the equation to blow-up the solution near a putative singularity and derive a
contradiction in the limit.  This method is quite flexible:  it has the advantage of yielding
certain conclusions under weaker assumptions on the data, and has therefore proven useful
for addressing such phenomena as the free boundary which arises in partial transport problems,
where the densities $f^\pm$ need not be continuous and are constrained but
not specified a priori \cite{CaffarelliMcCann10}.
Using this method, Caffarelli '92 was able to prove the following regularity result
on the {\em interior} $M^+_{int}$ of $M^+ \subset \R^n$:

\begin{thm}[Local regularity \cite{Caffarelli92}]\label{T:Caffarelli}
Fix $c=-x \cdot y$ and let $M^- \subset \R^n$ be convex.
Assume $\mu^\pm = f^\pm dx$
with $\log f^\pm \in L^\infty(M^\pm)$.  If $\log f^\pm \in W^{k,\infty}_{loc}(M^\pm_{int})$, there exists
$\alpha \in (0,1)$ (depending only on $n,k$ and the bounds on $\log f^\pm$)
such that $u \in C^{k+1,\alpha}_{loc}(M^+_{int})$,
where $u$ is the convex function with $Du(M^+) \subset M^-$ such that $Du_\# \mu^+ = \mu^-$.\\
 \end{thm}

\begin{rem}[Degenerate ellipticity] As shown also in \cite{Caffarelli92}, when one drops the convexity assumption on $M^-$ the gradient map may be discontinuous at interior points. This goes in hand with the claim made in the previous subsection that regularity (even in the interior) must propagate from the boundary.

\end{rem}

\begin{rem}[Local versus global regularity]\label{R:CaffarelliB}
This is a local regularity result. Global regularity (up to $\partial M^+$) requires
both domains $\partial M^\pm$ to be strongly convex
and smooth (Caffarelli '96) \cite{Caffarelli96b}.
Here {\em strong} convexity means the principal curvatures (or second fundamental form)
of the domain boundaries should be positive-definite.
\end{rem}

\begin{rem}[Higher regularity via uniformly elliptic linearization]\label{R:higher regularity}
The cases of primary relevance are $\log f^\pm$ merely bounded and measurable ($k=0$), and
$\log f^\pm$ also locally Lipschitz ($k=1$). Once $u \in C^{2,\alpha}_{loc}$ has been
deduced from these assumptions, higher regularity in the interior of $M^\pm$ follows from
uniform ellipticity of the Monge-Amp\`ere equation \eqref{Monge-Ampere}:
\begin{equation}\label{pre-linearized Monge-Ampere}
0=\log (\mbox{det}(D^2u(x)+\tau D^2w(x)))+\log (f^-(Du(x)+\tau Dw(x)))-\log (f^+(x)).
\end{equation}
For example, linearizing this equation at $\tau=0$ yields the equation
\begin{equation}\label{uniformly elliptic linearization}
0=Tr(D^2u(x)^{-1}D^2w(x))+D\log f^-|_{Du(x)} \cdot Dw(x),
\end{equation}
which must be satisfied by spatial derivatives $w=D_{x^i} u$ of $u$.
Convexity combines with $u \in C^{2,\alpha}_{loc}$ 
and the equation \eqref{Monge-Ampere} itself to bound $\|\Lambda^{n-1}f^-/f^+\|_{L^\infty}^{-1} \leq D^2u(x)\leq \Lambda$
on compact subsets of $M^+_{int}$. The derivatives of $u$ thus satisfy a uniformly elliptic
linear equation \eqref{uniformly elliptic linearization} with H\"older continuous coefficients,
so Schauder estimates \cite{GilbargTrudinger83}
and bootstrapping yield as much as regularity as can be expected when $k \geq 2$.
\end{rem}

\subsection{The continuity method (schematic)} (cf. Delano\"e '91, Urbas '97):
To apply the continuity method,
we assume $M^\pm \subset \R^n$ are smooth and strongly convex,
and $\log f^\pm \in C^{2,\alpha}(M^\pm) \cap L^\infty$.


Choose a dilation by $\epsilon>0$ sufficiently small and translation $M_0 =G_0(M^+)$ of $M^+$
by $x_0 \in \R^n$ such that $M_0 \subset\subset M^-$.
Let $\mu_0 := (G_0)_\#\mu^+$ be the push-forward of $\mu^+$
through the corresponding dilation and translation $G_0(x) = \epsilon x - x_0$.
Notice that $G_0$ is the gradient of the smooth convex function
$u_0(x) = \epsilon |x|^2/2 - x_0 \cdot x$, and as such gives the optimal map
between $\mu^+$ and $\mu_0$.  The idea behind the continuity method is to construct a family
of target measures $d\mu_t =f_t dVol$ interpolating between $d\mu_0$ and $d\mu_1 := d\mu^-$,
and to study the set $T$ of $t \in [0,1]$ for which the optimal transportation problem of Brenier
admits a solution with a convex potential $u_t \in C^{2,\alpha}(M^+)$.  The interpolating
measures must be constructed so that the $C^{2,\alpha}(M^\pm) \cap L^\infty$ norms
of $\log f_t$,  and the strong convexity and smoothness of $M_t = \spt f_t$,  can be
quantified independently of $t \in [0,1]$.  We then hope to show $T \subset [0,1]$ is both
open and closed.  If so,  it must exhaust the entire interval (since $0 \in T$), therefore
$1 \in T$ as desired.

Closed: To show closedness of this set requires an a priori estimate of the form
$\|u_t\|_{C^{2,\alpha}(M^+)} \le C(\|\log f^\pm\|_{C^{2,\alpha}(M^\pm)},
\|\partial M^\pm\|_{C^{2,\alpha}, {\rm strong\ convexity}})$
for any smooth solution $u_t \in C^4(M^+)$ of
the 2nd boundary value problem
\begin{equation}\label{2ndBVP}
\mbox{det} D^2u_t(x)=\frac{f_0(x)}{f_t(Du_t(x))}
\qquad {\rm with} \qquad Du_t(M^+) \subset M_t.
\end{equation}
Such estimates are delicate, but can be obtained by
differentiating the equation twice,  and constructing barriers.
Once obtained,  they imply that if $t_k \in T$ and $t_k \to t_\infty$ then $t_\infty \in T$ also;
the corresponding solutions $u_{t_k}$ belong to $C^4(M^+)$ as in Remark \ref{R:higher regularity}.

Open: The fact that $T$ is open is shown using an implicit function theorem in Banach spaces.
This requires knowing that the linearized operator is invertible (ie. uniformly elliptic),
and can be solved for the relevant boundary conditions.

For $u\in C^{2,\alpha}(M^+)$ we have already argued the uniform ellipticity
of the linearization \eqref{uniformly elliptic linearization}.  To linearize the boundary
conditions $Du_t(M^+) \subset M_t$,  introduce a sufficiently smooth and strongly convex function
$h: \R^n\to \R$ whose level sets $M_t= \{ y \in \R^n \mid  h(y)\leq t\}$ give the domains
$M_t:=\spt f_t$, and rewrite the non-linear boundary condition in the form $h(Du_t(x)) \le t$
with equality on $\partial M^+$.
Linearizing this in $u$ yields the boundary condition of the linear equation for $w$:
\begin{equation}\label{linear BC}
Dh(Du_t(x))\cdot Dw(x)=0 \qquad {\rm on} \ \partial M^+.
\end{equation}
For the linear problem \eqref{uniformly elliptic linearization} to be well-posed,
we need a {\em uniformly non-tangential} prescribed gradient for $w$ on $\partial M^+$.
Since $Dh$ parallels the normal $\hat n_{M_t}$ to $M_t$, this amounts to the uniform
obliqueness estimate
$$\hat{n}_{M_t}(Du_t(x))\cdot \hat{n}_{M^+}(x) \geq \delta>0\;\; \mbox{(obliqueness)}$$
provided by Urbas \cite{Urbas97}, with $\delta$ depending only on coarse bounds
for the data.  This concludes the sketch that $T\subset [0,1]$ is open:
well-posedness of the linear problem \eqref{uniformly elliptic linearization}--\eqref{linear BC}
when $t=t_0 \in T$
implies the existence of solutions $u_t \in C^{2,\alpha}(M^+)$ to the nonlinear problem
\eqref{2ndBVP} for any $t$ close enough to $t_0$.

Both approaches  (renormalization and continuity method) have been extended in recent years to more general costs, and this will be the topic of the next few lectures.\\

\section{Regularity and counterexamples for general costs}

\subsection{Examples}
The development of a regularity theory for general cost functions satisfying appropriate
hypotheses on compact domains $M^\pm \subset \R^n$ began with the work of Ma, Trudinger and
Wang \cite{MaTrudingerWang05}.  Prior to that there were regularity results only for a
few special costs, such as:

\begin{ex}[Bilinear cost]\label{E:bilinear}
$c(x,y)=-x\cdot y$ or equivalently $c(x,y) = |x-y|^2/2$
\cite{Delanoe91} \cite{Caffarelli92} \cite{Caffarelli96b} \cite{Urbas97},
and its restriction to $M^{\pm}=\partial B_1(0)$ in $\R^n$ (Gangbo and McCann \cite{GangboMcCann00});
\end{ex}

\begin{ex}[Logarithmic cost; conformal geometry and reflector antenna design]\label{E:logarithm}
$c(x,y)=-\log|x-y|$ appearing in conformal geometry
(cf. Viaclovsky's review \cite{Viaclovsky06p}),
and its restriction to the Euclidean unit sphere,
which is relevant to reflector antenna design
(Glimm and Oliker \cite{GlimmOliker03}, X.-J.~Wang \cite{Wang96}\cite{Wang04})
and helped to inspire Wang's subsequent
collaborations \cite{MaTrudingerWang05} \cite{TrudingerWang09b}
\cite{TrudingerWang09c} \cite{LiuTrudingerWang10}
with Trudinger, Ma, and Liu.
\end{ex}

In the wake of Ma, Trudinger and Wang's \cite{MaTrudingerWang05} results,  many new examples
have emerged of cost functions which satisfy \cite{Loeper08p}
\cite{DelanoeGe} \cite{FigalliRifford08p} \cite{FigalliRiffordVillani-MTW09p} \cite{FigalliRiffordVillani-Sn09p}
\cite{KimCounterexample} \cite{KimMcCann07p} \cite{KimMcCann08p} \cite{LeeLi09p} \cite{LeeMcCann09p} \cite{Lee09p}
\cite{Li09} \cite{LoeperVillani10} \cite{MaTrudingerWang05}
--- or which violate \cite{MaTrudingerWang05} \cite{Loeper09} --- their sufficient
conditions \Azero--\Afour\ and \Athree$_s$ for regularity from \S \ref{S:MTW conditions} below
 --- not to mention the subsequent variants \Athree\ and \Bthree\ introduced
by Trudinger and Wang \cite{TrudingerWang09b} and Kim and McCann \cite{KimMcCann07p} respectively,
on the crucial condition \Athrees.  Among the most interesting of these are the geometrical
examples and counterexamples of Loeper:

\begin{ex}[Sphere]\label{E:sphere}
$c(x,y)=\frac{1}{2}d^2_{\sphere^n}(x,y)$ on the round sphere satisfies \Athrees~\cite{Loeper08p}
      (and \Bthree~\cite{KimMcCann08p});
\end{ex}

\begin{ex}[Saddle]\label{E:saddle}
$c(x,y)=\frac{1}{2}d^2_{M}(x,y)$ on hyperbolic space $M=\hyperbolic^n$
      violates \Athrees\ (and \Athree\ \cite{Loeper09}) --- as does the Riemannian distance squared
      cost on any Riemannian manifold $M^\pm = M$ which has (at least one) negative sectional
      curvature at some point $x \in M$.
\end{ex}



\subsection{Counterexamples to the continuity of optimal maps}\label{S:counterexample}

For any cost function which violates \Athree,  Loeper went further to show
there are probability measures $d\mu^\pm = f^\pm d^n x$ with smooth positive
densities bounded above and below --- so that $\log f^\pm \in C^\infty\big(M^\pm\big)$ ---
for which the unique optimal map $G:M^+ \longrightarrow M^-$ is discontinuous \cite{Loeper09}.
Let's see why this is so for the quadratic cost given on either the hyperbolic plane or a saddle
surface as in Example \ref{E:saddle}.

Consider transportation from the uniform measure $\mu^+$ on a sufficiently small ball
to a target measure consisting of three point masses
$\mu^- = \frac{1}{3}\sum_{i\leq 3}\delta_{y_i}$ near the center of the ball,  choosing
$y_2$ to be the midpoint of $y_1$ and $y_3$.
In this case Theorem~\ref{T:GangboLevin} provides constants $v_1, \ldots, v_3$
and a $c$-convex function
\begin{equation}\label{3-point}
u(x) = \max \{u_i(x) \mid i =1,2,3\}\qquad \mbox{\rm where}\qquad u_i(x) = - c(x,y_i) - v_i,
\end{equation}
such that the optimal map $G_\# \mu^+ = \mu^-$ satisfies
$G^{-1}(y_i) = \{x \in M^+ \mid u(x) = u_i(x)\}
$.
We interpret $-v_i$ to be the value of the good at the potential destination $y_i \in \spt \mu^-$;
the producer at $x \in M^+$ will ship his good to whichever target point $y_i$ provides the
greatest value after transportation costs are deducted \eqref{3-point};
here the values $v_1,\ldots,v_3$ are
adjusted to balance supply with demand,  so that each of the three regions $G^{-1}(y_i)$
contains $1/3$ of the mass of $\mu^+$.  For the Euclidean distance-squared cost these three regions
are easily seen to be convex sets,  while for the spherical distance-squared they remain connected.
For the hyperbolic distance-squared, however,  the `middle' region $G^{-1}(y_2)$ consists of two
disconnected components, near opposite sides of the ball $\spt \mu^+$
(see figure below, or for instance Figure 1 of \cite{KimMcCann07p}).
This disconnectedness is the hallmark of costs for which
\Athree\ fails,  and allowed Loeper to construct 
counterexamples to the continuity of optimal mappings as follows.

In the preceding discussion,  $\mu^-$ was not given by a smooth positive density;
still it can be approximated by a sequence of measures $\mu^-_\epsilon := \mu^- * \eta_\epsilon$
which are.  Now consider the reverse problem of transporting $\mu^-_\epsilon$ to $\mu^+$.
Call the optimal map for this new problem $x = G^-_\epsilon(y)$.  For $\delta>0$,  taking
$\epsilon>0$ sufficiently small ensures for each $1 \le i \le 3$ that nearly $1/3$ of the mass of
$\mu^-_\epsilon$ concentrates near $y_i$ and is mapped into a $\delta$-neighbourhood of
$G^{-1}(y_i)$.  Intuitively, for $\delta$ sufficiently small, this forces a discontinuity of
$G^-_\epsilon$ which tears the region near $y_2$ into at least two disconnected components:
nearly half of the mass near this point must map to each disconnected component of $G^{-1}(y_2)$;
see Figure 2.
This construction shows why the distance-squared cost on a hyperbolic or saddle surface
cannot generally produce smooth optimal transport maps.

\begin{center}
\includegraphics[scale=0.30]{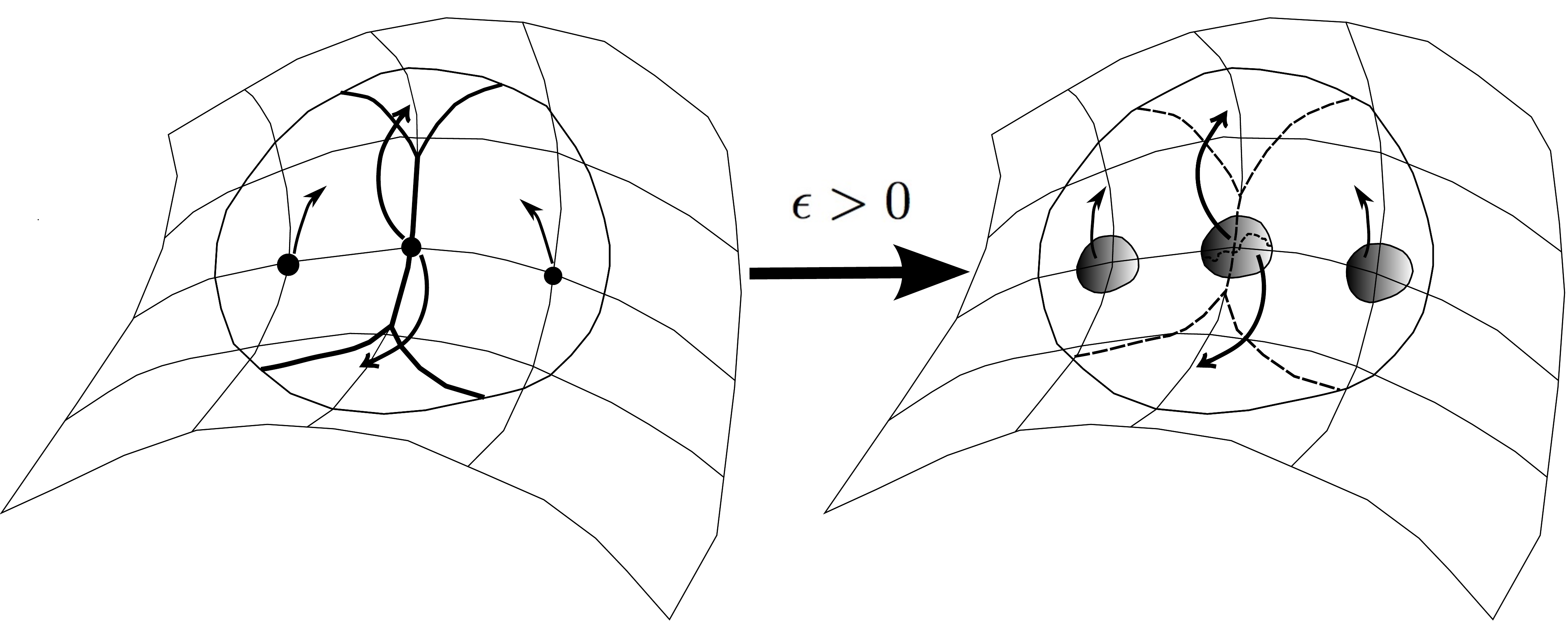}
\end{center}
\centerline{Figure 2: A tear occurs when spreading a triply peaked
density uniformly over the saddle.}
\smallskip

There is another of obstruction to the continuity of $G$,  namely the convexity
(at least when $M^\pm=\mathbb{R}^n$) of the support of $\mu^-$. This was shown by
Caffarelli \cite{Caffarelli92} with the following elementary example: consider
$u:\mathbb{R}^2\to\mathbb{R}$ given by
\begin{equation*}
    u(x)=|x_1|+\frac{1}{2}|x|^2,\;\;x=(x_1,x_2).
\end{equation*}
If we consider the cost $c(x,y)=-x\cdot y$, then $y=Du(x)$ gives the optimal transport map
between the unit disc (with Lebesgue measure) into two shifted half discs (Figure 3);
in particular, the transport map is discontinuous across $\{x_1=0\}$.

\bigskip
\begin{center}
\includegraphics[scale=0.30]{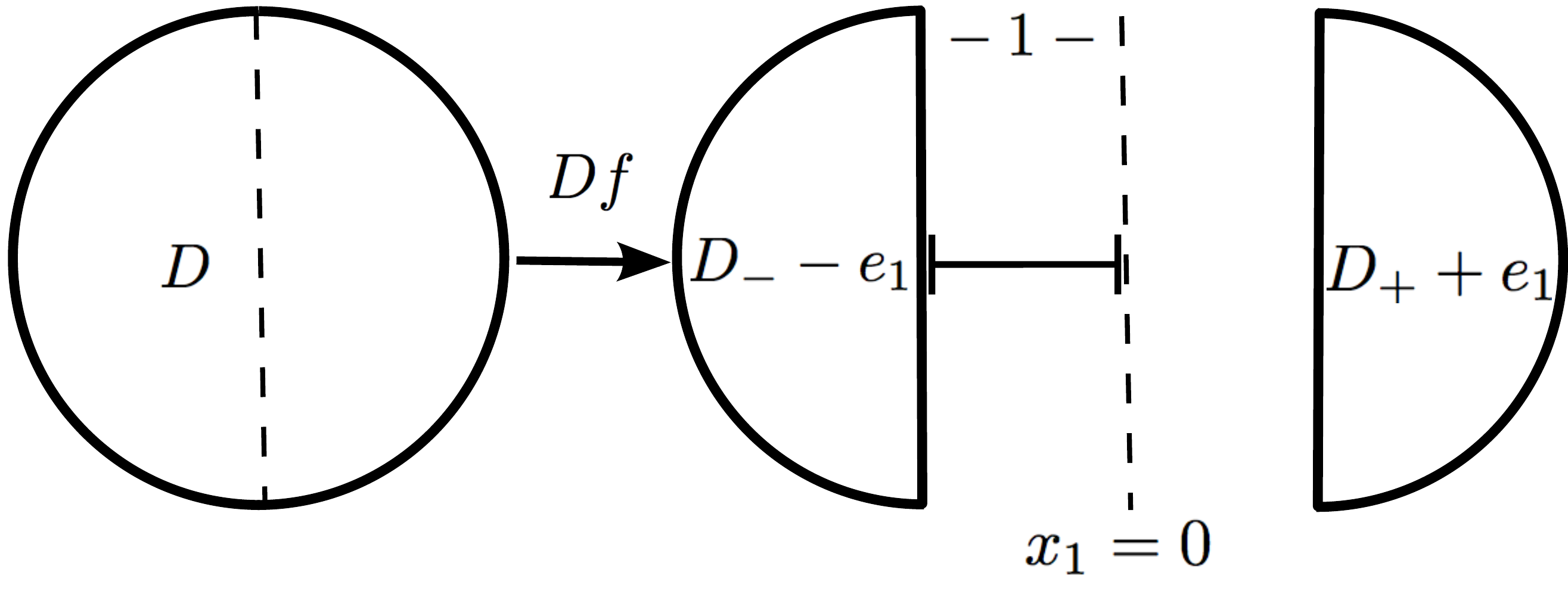}
\end{center}
\centerline{Figure 3. Disconnected targets also produce tears, as do non-convex targets ($\spt \mu^-$).}
\smallskip

We turn now to conditions which rule out these type of examples, and lead to positive regularity results.

\subsection{Monge-Amp\`ere type equations}

For the quadratic cost, finding a smooth optimal map was equivalent to solving the 2nd
boundary value problem for the Monge-Ampere equation \eqref{Monge-Ampere}.  Let us now
derive the analogous equation for a more general cost,  keeping in mind
that whatever PDE we end up with cannot generally be better than degenerate-elliptic,
since vanishing of $d\mu^- = f^- d Vol$ can lead to non-smooth solutions in the interior of the
support of $d\mu^+ =f^+ dVol$.

Let us see what specific PDE emerges from the local expression
$|\det DG(x)| = f^+(x)/f^-(G(x))$ for $G_\# \mu^+ = \mu^-$.
Recall from Corollary \ref{C:foc}
that $D^2_{xx}c(x,G(x))+D^2u(x) \geq 0$ and $D_x c (x,G(x))+Du(x)=0$.
Differentiating the latter expression gives a relation
$$
D^2_{xx}c(x,G(x)) +D^2_{xy}c(x,G(x))DG(x)+D^2u(x)=0
$$
which can be solved for $DG(x)$ to yield
\begin{equation}\label{MA type}
\det \left ( D^2u(x)+D^2_{xx}c(x,Y(x,Du(x))) \right ) =
\bigg| \det \left (D^2_{xy}c(x,y)\right ) \; \frac{f^+(x)}{f^-(y)} \bigg|_{y=Y(x,Du(x))}.
\end{equation}
Here we have assumed $\Aone^+$--\Atwo, the form $G(x) = Y(x,Du(x))$ of the optimal
map is from Theorem \ref{T:GangboLevin},  and the boundary condition is
$Y(x,Du(x)) \subset M^-$ for all $x \in M^+$. We have arrived as before at a fully-nonlinear
second-order equation, whose linearization around any $c$-convex solution $u=u^{\tc c}$ is
degenerate-elliptic.

\subsection{Ma-Trudinger-Wang conditions for regularity}\label{S:MTW conditions}

Sufficient conditions for the $c$-optimal map $G:M^+ \longrightarrow M^-$ to be smooth
between a pair of smooth bounded probability densities satisfying
$\log f^\pm \in C^\infty(M^\pm)$ on compact domains $M^\pm \subset \R^n$ were
found by Ma-Trudinger-Wang and Trudinger-Wang \cite{MaTrudingerWang05} \cite{TrudingerWang09b}.
The crucial condition on the cost $c(x,y)$ distinguishing
Examples \ref{E:bilinear}--\ref{E:sphere} from Example \ref{E:saddle} above
involves a quantity they identified,  which other authors have variously dubbed
the {\em Ma-Trudinger-Wang tensor }\cite{Villani09}, {\em $c$-sectional curvature} \cite{Loeper09},
or {\em cross-curvature} \cite{KimMcCann07p}; c.f.~\S \ref{S:differential geometry} below.
To define it, let us adopt their convention that subscripts such as
$c_{i,j} = \partial^2 c/\partial x^i \partial y^j$ and
$c_{ij,kl} = \partial^4 c/ \partial x^i \partial x^j \partial y^k \partial y^\ell$
indicate iterated derivatives in coordinates,  with commas separating
derivatives with respect to $x \in M^+$ from those with respect to variables $y \in M^-$.
Let $c^{j,i}(x,y)$ denote the inverse matrix to $c_{i,j}(x,y)$.

\begin{DEF}[Cross-curvature]
Given tangent vectors $p \in T_{x_0}M^+$ and $q \in T_{y_0}M^-$, define
$\cross(p,q):=\left (-c_{ij,kl}+c_{ij,r}c^{r,m}c_{m,kl} \right )p^ip^jq^kq^l$.
Here and subsequently, the Einstein summation convention is in effect.
\end{DEF}
%
%
The conditions assumed by Ma, Trudinger and Wang were the following \cite{MaTrudingerWang05};
our designations \Athree$_s$ and \Athree\ correspond to their \Athree\ and
{\rm (A3w)} from \cite{TrudingerWang09b}: \\

\Azero\ $c \in C^4(M^+\times M^-)$, and for all $(x_0,y_0)$ in the compact set
${M^+ \times M^-} \subset \R^n \times \R^n$;\\

\Aone\ $y \in {M^-} \longmapsto D_x c(x_0,y)$ {\rm\ and}\
$x \in M^+ \longmapsto D_y c(x,y_0)$\ {\rm are injective;} \\

\Atwo\ $\mbox{det}D^2_{x^i,y^j}c(x_0,y_0)=\mbox{det}(c_{i,j})\neq 0$;\\


\Athree\ $\cross(p,q) \ge 0$ for all $(p,q) \in T_{(x_0,y_0)} M^+ \times M^-$
such that $p^i c_{i,j} q^j =0$; \\

\Afour\ $M^-_{x_0}:=D_xc(x_0,M^-) \subset \R^n$ and $M^+_{y_0}:=D_yc(M^+,y_0) \subset T^*_{y_0}M^-$
are convex. \\
%

Among the variants on \Athree\ subsequently
proposed \cite{KimMcCann07p} \cite{LoeperVillani10} \cite{FigalliRifford08p},
let us recall the {\em non-negative cross-curvature} condition \cite{KimMcCann07p}:\\

\Bthree\ $\cross(p,q) \ge 0$ for all $(p,q) \in T_{(x_0,y_0)} M^+ \times M^-$. \\

The first two conditions above are familiar from Theorems \ref{T:GangboLevin} and
\ref{T:McCannPassWarren};  \Aone\ was proposed 
independently of \cite{MaTrudingerWang05}
in \cite{Gangbo95} \cite{Levin99}, while there is an antecedent for \Atwo\ in
the economics literature \cite{McAfeeMcMillan88}.  The last condition \Afour\ adapts
the convexity required by Delano\"e, Caffarelli (Theorem \ref{T:Caffarelli}) and Urbas,
to the geometry of the cost function $c(x,y)$;  when $M^+_{x_0}$ and $M^-_{y_0}$ are smooth and their
convexity is strong --- meaning the principal curvatures of their boundaries are all strictly
positive --- we denote it by \Afour$_s$. When inequality \Athree\ or \Bthree\ holds strictly ---
and hence uniformly on the compact set $M^+ \times M^-$ --- we denote that fact by
\Athree$_s$\ or \Bthree$_s$, respectively.

\begin{rem}\label{R:bilinear borderline}
The quadratic cost $c(x,y) = -x \cdot y$ of Brenier
satisfies \Bthree\ but not \Athree$_s$.  Since we have already seen that
\Athree\ is necessary \cite{Loeper09} as well as sufficient for the continuity
of optimal maps,  the quadratic cost is actually a delicate borderline case.
The negative $-c$ of any cost $c$ satisfying \Athree$_s$ --- including those of Examples
\ref{E:logarithm}--\ref{E:sphere} --- necessarily violates \Athree.
\end{rem}

\subsection{Regularity results}

Assume \Azero--\Afour\ and $\log f^\pm \in C^\infty(M^\pm)$.  Under the stronger condition
\Athree$_s$,  Ma, Trudinger and Wang \cite{MaTrudingerWang05}
proved the interior regularity of the optimal map $G$
and corresponding $c$-convex potential $u \in C^\infty(M^+_{int})$;
a flaw in their argument was later repaired in \cite{TrudingerWang09c}
(see \cite{KimMcCann07p} for another approach).
Substituting strong convexity \Afour$_s$ for \Athree$_s$,  but retaining \Athree,
Trudinger and Wang \cite{TrudingerWang09b} used the continuity method to establish
regularity up to the boundary $u \in C^\infty(M^+)$.  Relaxing strong convexity to
\Afour\ in that context is an open problem.

For densities merely satisfying
$f^+/f^- \in L^\infty(M^+ \times M^-)$,  under the strong condition
\Athree$_s$, Loeper  was able to establish local H\"older continuity of the optimal map
--- or equivalently $u \in C^{1,\alpha}_{loc}(M^+_{int})$ --- with explicit H\"older exponent
$\alpha = 1/(4n-1)$, using a direct argument \cite{Loeper09} that we sketch out below.
This exponent was later improved to its sharp value $\alpha = 1/(2n-1)$ by Liu \cite{Liu09}.
For the quadratic cost $c(x,y) = -x\cdot y$,  the best known estimates \cite{ForzaniMaldonado04}
for the H\"older exponent $\alpha$ are much worse, and depend on bounds for $\log (f^+/f^-)$.
Assuming non-negative cross-curvature
\Bthree\ and \Afour$_s$ instead of \Athree$_s$, Figalli, Kim and McCann adapted Caffarelli's
renormalization techniques \cite{FigalliKimMcCann} to derive
continuity and injectivity of optimal maps but without any H\"older exponent;
using one of their arguments,
a similar conclusion was obtained by Figalli and Loeper  \cite{FigalliLoeper09}
in the special case $n=2$ assuming only \Athree\ and \Afour$_s$.
Liu, Trudinger and Wang showed that higher
regularity then follows from further assumptions on $f^+/f^-$ in any
dimension \cite{LiuTrudingerWang10}.

Using this theory, regularity results have now been obtained in geometries such
as the round sphere \cite{Loeper08p},  perturbations \cite{DelanoeGe} \cite{FigalliRifford08p}
\cite{FigalliRiffordVillani-Sn09p}, submersions \cite{DelanoeGe} \cite{KimMcCann08p}
and products \cite{FigalliKimMcCann-spheres} thereof, and hyperbolic space
\cite{Li09} \cite{LeeLi09p}. Significant cut-locus issues arise in this context.
Loeper and Villani  \cite{LoeperVillani10} conjecture,  and in some cases have proved,
that condition \Athree$_s$ on the quadratic cost $c(x,y)=d^2(x,y)$ 
actually implies convexity of the domain of injectivity of the Riemannian exponential map
$\exp_x : T_x M \longrightarrow M$.

\subsection{Ruling out discontinuities: Loeper's maximum principle}

Let us discuss how the condition (A3) rules out the tearing phenomenon which we saw
on the saddle surface of Example \ref{E:saddle}.

Discontinuities in the optimal map $G(x) = Y(x,Du(x))$ correspond to locations $x_0 \in M^+$ where
differentiability of the potential function $u=u^{\tc c}$ fails, such as locations
where the supremum
\begin{equation}\label{c0-transform}
u(x_0)=\sup \limits_{y \in M^-} -c(x_0,y)-u^\tc(y)
\end{equation}
is attained by two or more points $y_0 \ne y_1 \in M^-$.
\begin{center}
\includegraphics[scale=0.30]{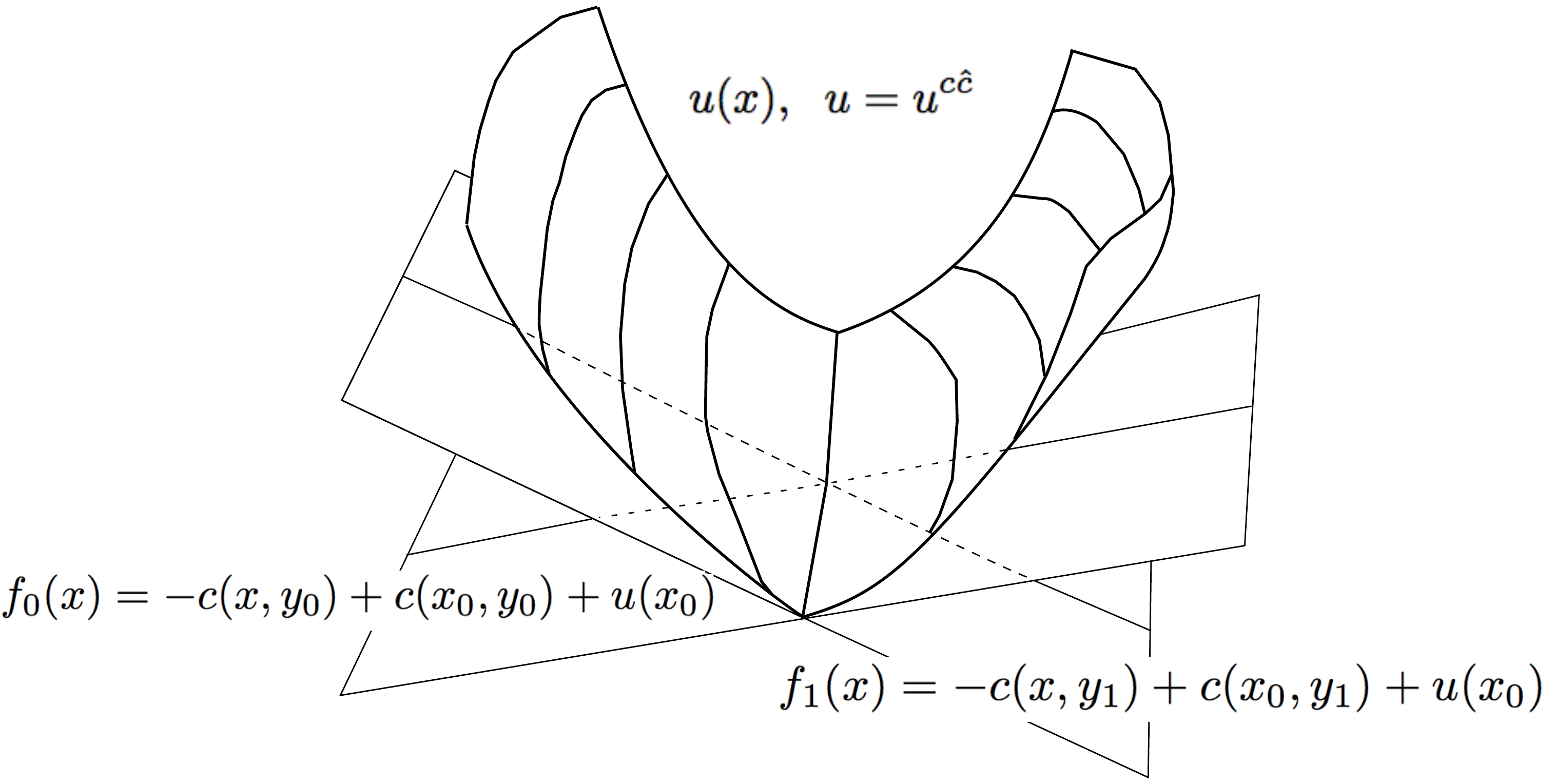}
\end{center}
\centerline{Figure 4. Discontinuous optimal maps arise from distinct supporting hyperplanes.}
\smallskip \noindent
The set of such $y$ is denoted by $\partial^c u(x_0)$,  while the
set of such pairs is denoted $\partial^c u \subset M^+ \times M^-$.
Unless we can find a continuous curve
$t \in [0,1] \longmapsto y_t \in \partial^c u(x_0)$
which connects $y_0$ to $y_1$,
it will be possible \cite{Loeper09} to construct probability densities satisfying
$\log f^\pm \in C^\infty(M^\pm) \cap L^\infty$ with a discontinuous optimal map as in
\S \ref{S:counterexample} above.  But there are not many possibilities to have such
a curve.

In the classical case $c(x,y)=-x\cdot y$ (see Figures 4 and 5),
$\partial^c u = \partial u$ and we have a convex function $u$ with
two different supporting planes at $x_0$.
%

\begin{center}
\includegraphics[scale=0.30]{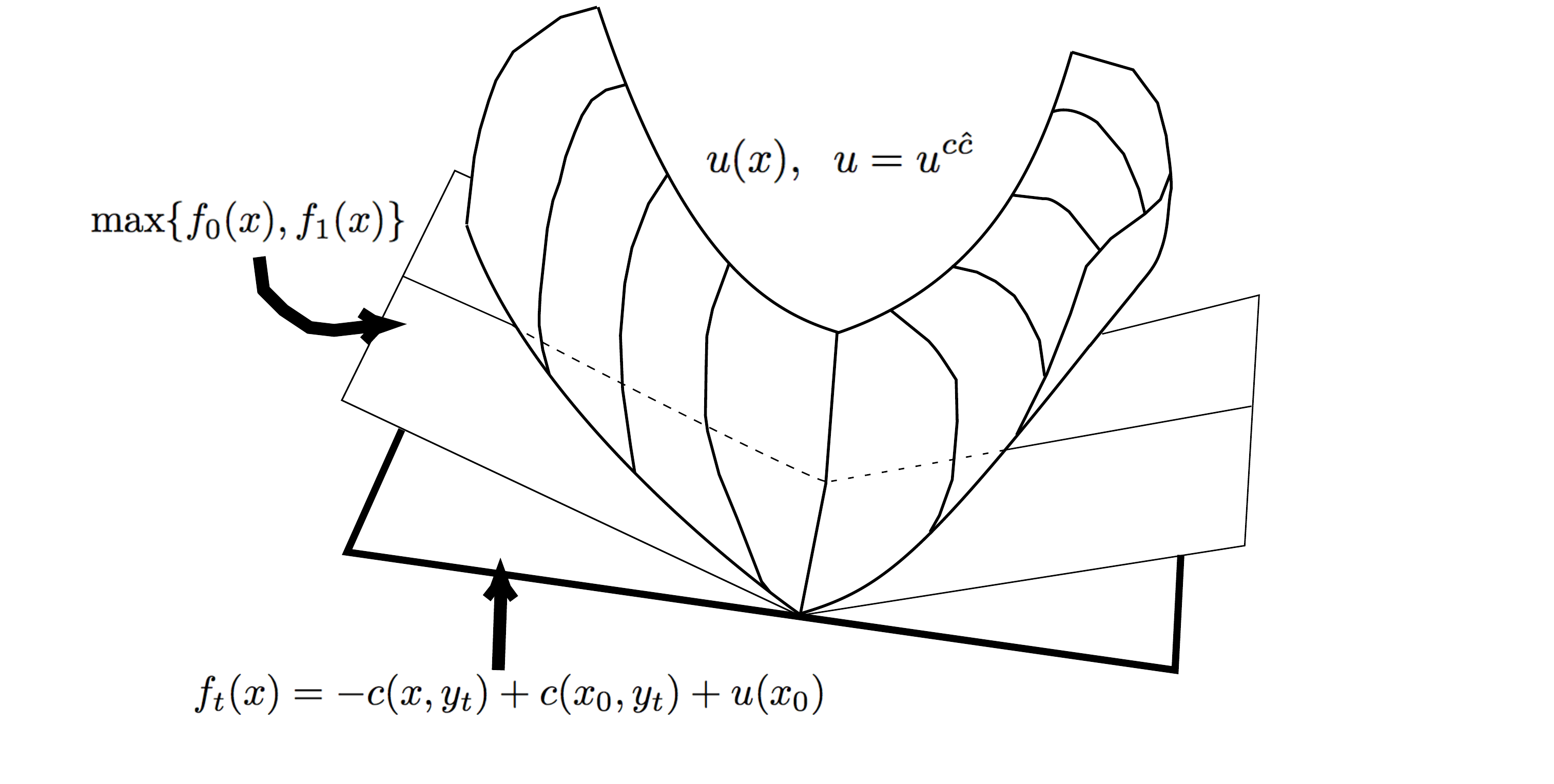}
\end{center}
\centerline{Figure 5. Can one $c$-affine support at $x_0$ be rotated to another, without exceeding $u$?}

\smallskip\noindent
In particular one may continuously rotate the first plane about the point $(x_0,u(x_0))$
without ever crossing the graph of $u$ until it agrees with the second plane giving a
one-parameter family of supporting planes to $u$ at the same point. This way one sees
that $\partial u({x_0})$ contains a ``segment'' $\{y_t\}_{t\in(0,1)}$. In this special
case $y_t=(1-t)y_0+ty_1$ where $y_0$ and $y_1$ are the slopes of the original supporting
hyperplanes.  In the general case, the corresponding local picture forces the
``c-segment'' $\{y_t\}_{t\in(0,1)}$ given by
\begin{equation}\label{c-segment}
D_xc(x_0,y_t)=(1-t)D_xc(x_0,y_0)+tD_xc(x_0,y_1)
\end{equation}
to be our only hope for a continuous path connecting $y_0$ to $y_1$ in $\partial^c u(x_0)$.

Now, were $G$ to exhibit a discontinuity, this construction suggests $G$ has to transport a very small mass around $x_0$ into a set with very large mass, which would give a contradiction given the constraint $G_\# \mu^+ = \mu^-$ and the assumptions on $\mu^\pm$. This is indeed the case, at least under the stronger assumption \Athree$_s$, as we shall see below (Proposition \ref{P:sausage-ball} and Theorem \ref{T:Loeper}).

All of this is of course contingent on whether the entire family of functions $\{g_t\}_{t}$
lie below $u(x)$,
which might not be true for an arbitrary cost $c$; see Figure 5.
Indeed, Loeper's key observation is that the
Ma-Trudinger-Wang condition \Athree\ is what guarantees that any family of functions
$f_t(y)=-c(x,y_t)+c(x_0,y_t)$ with $y_t$ satisfying (\ref{c-segment}) never goes above $u(x)$.
More precisely, it remains below $\max\{f_0(x),f_1(x)\}$.

\begin{thm}[Loeper's maximum principle \cite{Loeper09} \cite{KimMcCann07p}]\label{T:DASM}
If \Azero--\Afour\ hold and $x_0 \in M^+$ and $(y_t)_{t \in[0,1]} \subset M^-$
satisfy \eqref{c-segment}, then
\begin{equation}\label{DASM}
f(x,t) :=-c(x,y_t)+c(x_0,y_t)\leq \max \{ f(x,0),f(x,1)\} \quad \forall\ (x,t) \in M^+ \times [0,1].
\end{equation}
\end{thm}

\begin{rem}(\cite{KimMcCann08p})\label{R:convex DASM}
If in addition, \Bthree\ holds,  then $t \in [0,1] \longrightarrow f(x,t)$ is convex .
\end{rem}

Loeper's original proof was quite tortuous,  relying on global regularity results
for optimal transportation already established by Trudinger and Wang \cite{TrudingerWang09b}.
Here we sketch instead a simple, direct proof due to Kim and McCann \cite{KimMcCann07p},
who later added Remark \ref{R:convex DASM}.
A preliminary lemma gives some insight into the relevance of the cross-curvature.

\begin{lem}[A non-tensorial expression for cross-curvature \cite{KimMcCann07p}]\label{L:cross-curvature}
Assuming \Azero--\Afour, if
$(x(s))_{-1\leq s\leq 1} \subset M^+$ and $(y(t))_{-1\leq t\leq 1} \subset M^-$ satisfy either
$$
\frac{d^2}{dt^2}\bigg|_{t=0} D_xc(x(0),y(t))=0 \qquad {\rm or} \qquad
\frac{d^2}{ds^2}\bigg|_{s=0} D_yc(x(s),y(0))=0,
$$
then $\cross 
(\dot{x}(0),\dot{y}(0))=-\frac{\partial^4}{\partial^2s\partial^2t} \Big|_{s=0=t} c(x(s),y(t)).$
\end{lem}

This lemma is precisely analogous to the formula for the distance between two arclength
parameterized geodesics $x(s)$ and $y(t)$ passing through $x(0)=y(0)$ in a Riemannian manifold:
$$
d^2(x(s),\bar x(t)) = s^2 + t^2 - 2 st \cos\theta - \frac{k}{3} s^2 t^2 \sin^2\theta + O((s^2+t^2)^{5/2})
$$
where $\theta$ is the angle between $\dot x(0)$ and $\dot y(0)$ and $k$ is the sectional curvature
of the plane which they span.  Therefore, we will not give its proof.

\begin{proof}[Proof of  Remark \ref{R:convex DASM} and sketch of Theorem \ref{T:DASM}]

    Assume \Athree$_s$ for simplicity. It suffices to prove the following claim. \\

    Claim 1: if $\frac{\partial f}{\partial t}(x,t_0)=0$ then $\frac{\partial^2 f}{\partial t^2}(x,t_0)>0$. \\

{\em Proof of Claim 1:}  Convexity \Afour\ allows us to define $s \in [0,1] \longmapsto x(s)$ by
\begin{equation}\label{s-segment}
D_y c(x(s),y(t_0))=(1-s)D_yc(x_0,y(t_0))+sD_yc(x,y(t_0))
\end{equation}
and $g(s)=\frac{\partial^2 f}{\partial t^2}(x(s),t_0)$. Our claim is that $g(1)>0$.
    Since $f(x_0,t) = 0$ and hence $g(0)=0$,  to prove Claim 1 it suffices to establish
    strict convexity in Claim 2. \\

    Claim 2: $g:[0,1]\to\R$ is convex, and minimized at $s=0$.\\

{\em Proof of Claim 2:} Once $g(s)$ is known to be convex, we need only observe that
$$
g'(0)=-\frac{\partial^3}{\partial s\partial t^2}\bigg|_{s=0,t=t_0}c(x(s),y(t))
$$
vanishes by our choice \eqref{c-segment} of $y(t)=y_t$, to conclude $g(s)$ is minimized at $s=0$.

    Why should $g(s)$ be convex?  Note that
    $g''(s)=-\frac{\partial^4 }{\partial s^2\partial t^2}\Big|_{t=t_0}c(x(s),y(t))$
    is already non-negative according to Lemma \ref{L:cross-curvature} if we assume \Bthree.
    Remark \ref{R:convex DASM} is thereby established.  Under the weaker condition \Athree,
    we need $\dot{x}^i(s)c_{i,j}\dot{y}^j(t_0)=0$
to conclude $g(s)$ is convex --- and strictly convex if \Athree$_s$ holds.
But
$$
0 = \frac{\partial f}{\partial t}(x,t_0) = \int_0^1 c_{i,j}(x(s),y(t_0)) \dot x^i(s) \dot y^j(t_0) ds
$$
and the integrand is constant by our construction \eqref{s-segment} of $x(s)$.
\end{proof}

To deduce the continuity result of the next section,  the following corollary is crucial.

\begin{cor} 
\label{C:local-global}
Assume \Azero--\Afour\ and fix $(x_0,y_0) \in M^+_{int} \times M^-$.  If $u=u^{\tc c}$ satisfies
\begin{equation}\label{local support}
u(x) \ge -c(x,y_0) + c(x_0,y_0) + u(x_0)
\end{equation}
in a neighbourhood of $x_0$,  the same equality holds for all $x \in M^+$.
\end{cor}

\begin{proof}
The local inequality \eqref{local support} implies $p_0 := -D_x c(x_0,y_0) \in \partial u(x_0)$.
If $x_0 \in \dom Du$,  the conclusion is easy.  The global inequality
$$
u(x) \ge -c(x,y_1) + c(x_0,y_1) + u(x_0)
$$
holds for any $(x_0,y_1) \in \partial^c u$,  and for
$y_1 \in \arg\max_{y \in M^-} -c(x_0,y) - u^\tc(y)$ in particular.
The twist condition \Aone\ then implies $y_0=y_1$.

Even if $x_0 \not\in \dom Du$, taking e.g.\ $p = -D_x c(x_0,y_1)$ yields
\begin{equation}\label{levelsetconvex}
-c(x,Y(x_0,p)) + c(x_0,Y(x_0,p)) \le u(x) - u(x_0) \qquad \forall x \in M^+.
\end{equation}
In fact, the set $P =\{ p \in M^-_{x_0} \mid
\eqref{levelsetconvex} {\rm\ holds}\}$ is convex,  according to Theorem \ref{T:DASM}.
On the other hand,  $P$ includes all the extreme points $p$ of $\partial u(x_0)$,
since the preceding argument can be applied to a sequence
$(x_k,y_k) \in \partial^c u \cap (\dom Du \times M^-)$ with
$(x_k,Du(x_k)) \to (x_0,p)$.  Thus $P \supset \partial u(x_0)$,
whence $p_0 \in P$ as desired.  (In fact, $P = \partial u(x_0)$).
\end{proof}

\subsection{Interior H\"older continuity for optimal maps}

To conclude our discussion on regularity of optimal mappings,
let us sketch Loeper's H\"older continuity result \cite{Loeper09}.

\begin{thm}[Loeper '09]\label{T:Loeper}
Assume \Azero--\Atwo and \Afour.
(i) If \Athree\ is violated,  there exist
probability densities with $\log f^\pm \in C^\infty(M^\pm) \cap L^\infty$ and
a discontinuous optimal map $G:M^+_{int} \longrightarrow M^-$ satisfying $G_\#(f^+ dVol) = f^- dVol$.
(ii) Conversely, if \Athree$_s$ holds and $f^+/f^- \in L^\infty(M^+ \times M^-)$, then
$G \in C^{\frac{1}{4n-1}}_{loc}(M^+_{int},M^-)$.
\end{thm}









In one dimension $n=1$,  we see $G$ is Lipschitz directly from the
equation $G'(x) = f^+(x)/f^-(G(x))$.  In higher dimensions,
this theorem is a direct consequence of the following proposition,
whose inequalities $\sim$ hold up to multiplicative constants depending only on the cost $c$,
and in particular on the size of the uniform modulus of positivity in condition \Athree$_s$.

\begin{prop}[Sausage into ball \cite{Loeper09}]\label{P:sausage-ball}
Assuming the hypotheses and notation of Theorem \ref{T:Loeper}(ii),
take $x_0,x_1 \in M^+$ and set $\Delta x = x_1-x_0$ and $\Delta y = y_1 - y_0$
where $y_i = G(x_i)$.
If $|\Delta x|\lesssim |\Delta y|^{5}$,
there is a ball $B_\epsilon(x) \supset G^{-1}(S_\delta)$
of radius $\epsilon \sim \sqrt{\frac{|\Delta x|}{|\Delta y|}}$
centered on the line segment joining $x_0$ to $x_1$, containing the preimage of the ``sausage''
$$S_\delta = \{ y \in M^-\;| \inf\limits_{t \in [1/3,2/3]}|y-y_t|\leq \delta \}$$
of radius $\delta \sim \epsilon |\Delta y|^2$ around the middle third
of the curve $(y_t)_{t \in [0,1]} \subset M^-$ satisfying $0=\frac{d^2}{dt^2} D_x c(x,y_t)$.
\end{prop}

\begin{proof}[Proof of Theorem \ref{T:Loeper}(ii)]
At pairs of points $y_i = G(x_i)$ where $\|\Delta x\| \gtrsim \|\Delta y\|^5$ we already
have H\"older exponent $1/5$ --- even better than claimed.  At other points, using the
fact that $G$ is a transport map between $\mu^\pm = f^\pm dx$, the Proposition yields
$\mu^-(S_\delta)=\mu^+(G^{-1}(S_\delta)) \lesssim \|f^+\|_\infty \epsilon^n$,
but also $\delta^{n-1}|\Delta y| \inf\limits_{M^-}f^- \lesssim \mu^-(S_\delta)$.
Combining the squares of these two inequalities,  our choices
$\delta \sim \epsilon | \Delta y|^2$ and $\epsilon^2 \sim |\Delta x|/|\Delta y|$ yield
the desired H\"older estimate:
$$
\|(f^-)^{-1}\|^{-2}_\infty \epsilon^{2n-2}|\Delta y|^{4n-2}\lesssim \|f^+\|^2_\infty \epsilon^{2n-2} \frac{|\Delta x|}{|\Delta y|}.
$$
Thus $G\in C^{\frac{1}{4n-1}}_{loc}$.
\end{proof}

The proposition relies delicately on Corollary \ref{C:local-global}
and the correct choice of $\delta$ and $\epsilon$:

\begin{proof}[Proof sketch of Proposition \ref{P:sausage-ball}; c.f. \cite{KimMcCannAppendices}:]
According to Theorem \ref{T:GangboLevin}, the optimal map $G(x) = Y(x,Du(x))$
is given by a potential $u=u^{\tc c}$ and $Graph(G) \subset \partial^c u$.
Thus $(x_i,y_i) \in \partial^c u$,  meaning
$f_i(x) = -c(x,y_i) + c(x_i,y_i) + u(x_i)$
satisfies $u(x) \ge \max \{f_0(x),f_1(x)\}$ with equality at $x_0$ and $x_1$.
Take $x$ to be the point on the segment joining
$x_0$ to $x_1$ where $f_0(x) = f_1(x)$ ($=0$ without loss of generality).
The semiconvexity of $u$ shown in Lemma~\ref{L:Lipschitz} then yields the bound
$u(x) \lesssim |\Delta x||\Delta y| + |\Delta x|^2$.
Assumption \Athree$_s$ allows Theorem~\ref{T:DASM} to be quantified,  so that
$f_t(\cdot) := -c(\cdot,y_t) + c(x,y_t) + u(x)\le u(\cdot)$ actually satisfies
$$f_t(x') - u(x') \lesssim -t(1-t)|x' -x|^2|\Delta y|^2
$$
for $x'$ near $x$.  For $t \in [1/3,2/3]$,  these estimates give some leeway to shift
$y_t$ up to distance $\delta$ without spoiling the inequality
$g_y(x') := -c(x',y) + c(x,y) + u(x) \le u(x')$ on the boundary
$x' \in \partial B_\epsilon(x)$. Since $g_y(x) = u(x)$, this inequality does not extend
to the interior of the ball $B_\epsilon(x)$,  unless we subtract some non-negative
constant from $g_y(\cdot)$.
Subtracting the smallest such constant $\lambda$ yields a function
$g_y(\cdot) - \lambda \le u(\cdot)$ on $B_\epsilon(x)$,
with equality at some $x_* \in B_\epsilon(x)$.  Corollary \ref{C:local-global}
implies $(x_*,y) \in \partial^c u$.  For almost every such $y \in S_\delta$
this provides the desired preimage $x_* \in G^{-1}(y)$.
\end{proof}

\section{Multidimensional screening: an application to economic theory}

We now sketch an application \cite{FigalliKimMcCann-econ} of the mathematics we
have developed to one of the central problems in microeconomic theory:
making pricing or policy decisions for a
monopolist transacting business with a field of anonymous
agents whose preferences are known only statistically.
Economic buzzwords associated with problems
of this type include ``asymmetric information,'' ``mechanism design,'' ``incentive compatibility,''
``nonlinear pricing,'' ``signalling,'' ``screening,''
and the ``principal / agent'' framework.

\subsection{Monopolist nonlinear pricing and the principal-agent framework}
To describe the problem,  imagine we are given: a set of ``customer'' types $M^+ \subset \R^n$
and ``product'' types $M^- \subset \R^n$ and \\

$b(x,y)$= benefit of product $y \in M^-$ to customer $x \in M^+$;\\

$a(y)$= monopolist's cost to manufacture $y \in M^-$;\\

$d\mu^+(x)\geq 0$ relative frequency of different customer types on $M^+$.\\

\noindent

Knowing all this data,  the \textbf{Monopolist's problem} is to assign a price
to each product,  for which she will be willing to
manufacture that product and sell it to whichever agents choose to buy it.
Her task is to design the price menu $v:M^- \to\R\cup \{+\infty\}$ so as to maximize
profits. The only constraint that prevents her from raising prices arbitrarily high
is the existence of a fixed $y_\emptyset \in M^-$, called the ``outside option'' or ``null product'',
which she is compelled to sell at cost $v(y_\emptyset)=a(y_\emptyset)$. Though it is not necessary,
we can fix the cost of the null product to vanish without loss of generality.

The \textbf{Agent's problem} consists in computing
\begin{equation}\label{indirect utility}
u(x)=\max\limits_{y \in M^-} b(x,y)-v(y)
\end{equation}
and
choosing to buy that product $y_{b,v}(x)$ 
for which the maximum is attained. The monopolist is generally called the {\em principal}, while the
customers are called {\em agents}.
\\

Economists use this framework to model many different types of transactions,
including tax policy  \cite{Mirrlees71} (where the government wants to decide a tax structure
which encourages people both to work and report income),
contract theory  \cite{Spence74} (where a company wants to decide a salary structure which
attracts and rewards effective employees without overpaying them),
and the monopolist nonlinear pricing problem described above \cite{MussaRosen78}.
In the initial studies, the type spaces $M^\pm$ were assumed one-dimensional, 
with $x \in M^+$ representing the innate ability or talent of the
prospective tax-payer or employee,  and $y \in M^-$ the amount of work that he
chooses to do or the credentials he chooses to acquire.  The basic insight of
Mirrlees and Spence was that under condition \Atwo\ (which implies \Aone\ in
a single dimension)  the variables $x,y \in \R$ would be monotonically correlated by
the optimal solution,  reducing the monopolist's problem to an ordinary differential
equation.  For this reason the one-dimensional versions of \Aone--\Atwo\ are called
{\em Spence-Mirrlees} (or {\em single-crossing}) conditions in the economics literature;
both Mirrlees and Spence were awarded Nobel prizes for exploring the economic implications
of their solution.

Of course,  many types of products are more realistically modeled
using several parameters $y \in \R^n$ --- in the case of cars these might include
fuel efficiency, size, comfort, safety, reliability, and appearance ---
while the preferences of customers for such parameters are similarly nuanced.
Thus it is natural and desirable to want to solve the multidimensional version
$n \ge 2$ of the problem,  about which much less is known \cite{Basov05}.
Monteiro and Page \cite{MonteiroPage98} and independently Carlier \cite{Carlier01}
showed only that enough compactness remains to conclude that the monopolist's optimal strategy
exists. An earlier connection to optimal transportation can be discerned in the
work of Rochet \cite{Rochet87},  who proved a version of Theorem \ref{T:Rockafellar} (Rockafellar)
for general utility functions $b$ ($=-c$ in our earlier notation).




Rochet and Chon\'e \cite{RochetChone98} studied the special case $b(x,y)=x\cdot y$
on $M^\pm = [0,\infty[^n$.
Taking $a(y)=\frac{1}{2}|y|^2$, $d\mu^+=\chi_{[0,1]^2} d^2 x$, and $y_\emptyset=(0,0)$, they
deduced that the mapping
$
y_{b,v}:M^+\to M^-
$
was the gradient of a convex function;  it sends
a positive fraction of the square to the point mass $y_\emptyset$,
and a positive fraction to the line segment $y_1=y_2$,
while the remaining positive fraction gets mapped in a bijective manner, so that
\begin{equation}\label{bunching}
\mu^- :=(y_{b,n})_{\#}\mu^+=f^-_0\delta_{y_\emptyset}+f^-_1d\HM^1+f^-_2d\HM^2.
\end{equation}
They interpreted this solution to mean that while the top end of the market
gets customized vehicles $f^-_2$,  price discrimination alone
forces those customers in the next market segment to choose from a more limited set $f^-_1$
of economy vehicles offering a compromise between attributes $y_1$ and $y_2$.  A fraction
$f^-_0>0$ of consumers will be priced out of the market altogether ---
which had already been observed by Armstrong \cite{Armstrong96} to be a hallmark
of nonlinear pricing in more than one dimension $n \ge 2$ .  Economists refer to this
general phenomenon \eqref{bunching} as ``bunching'',
and to the fact that $f^-_0>0$ as ``the desirability of exclusion.''

How robust is this picture?  It remains a pressing question to understand whether the
bunching phenomena of Rochet and Chon\'e is robust, or merely an accident of the particular
example they explored.  As we now explain,  their results were obtained by reducing the
monopolist's problem to the minimization of a Dirichlet energy:
\begin{equation}\label{Dirchlet enery}
\min_{ 0\leq u {\rm\ convex}}
\int_{[0,1]^2}\left ( \frac{1}{2}|Du|^2- \langle x,Du(x) \rangle +u(x)\right) d\HM^2(x).
\end{equation}
The constraint that $u:M^+ \longrightarrow \R$ be {\em convex} makes this problem non-standard:
its solution satisfies a Poisson type equation only on the set where $u$ is strongly convex
($D^2 u > 0$),
and there are free boundaries separating the regions where the different constraints $u \ge 0$
and $D^2 u \ge 0$ begin to bind.

\subsection{Variational formulation using optimal transportation}

The principal's problem is to
choose $v:M^- \longrightarrow \R \cup \{+\infty\}$
to
maximize her profits,  or equivalently to minimize
her net losses:
\begin{equation}\label{principal's naive problem}
\min_{\{ v \mid v(y_\emptyset)=a(y_\emptyset)\}} \int_{M^+} [a(y_{b,v}(x))-v(y_{b,v}(x)) ]d\mu^+(x).
\end{equation}
Note that the integrand vanishes for all customers $x$ who choose the null product $y_{b,v}(x)=y_\emptyset$.

Wherever the agent's maximum 
\eqref{indirect utility} is achieved, we have $Du(x)-D_xb(x,y_{b,v}(x))=0$,
so using the twist condition \Aone\ from our previous lectures we can invert this relation
to get $y_{b,v}(x)=Y(x,Du(x))$.  Moreover,  the function $u(x)$ from \eqref{indirect utility}
is a $b$-convex function, called the {\em surplus} or {\em indirect utility} $u = u^{\tb b}$.
(In our previous notation,  $u$ is a $(-b)$-convex
function and $u=u^{(\tilde{-b})(-b)}$,  but we suppress the minus signs hereafter.)

Since $v(Y(x,Du(x)) = b(x,Y(x,Du(x)) - u(x)$, 
we may reformulate
the variational problem \eqref{principal's naive problem} as the minimization of the principal's losses
\begin{equation}\label{principal's losses}
L(u) := \int_{M^+}\left [a(Y(x,Du(x)))-b(x,Y(x,Du(x))) + u(x)\right ]d\mu^+(x).
\end{equation}
over the set $\mathcal{U}_\emptyset =\{u \in \mathcal{U} \mid u \ge u_{\emptyset} \}$
of $b$-convex functions $\mathcal{U}=\{u \mid u=u^{\bar b b} \}$ which
exceed the reservation utility $u_\emptyset (\cdot )= b(.,y_\emptyset)-a(y_\emptyset)$
associated with the outside option or null product.
This strange reformulation due to Carlier \cite{Carlier01}
\begin{equation}\label{Carlier's problem}
\min_{u \in \mathcal{U}_\emptyset} L(u)
\end{equation}
reduces to \eqref{Dirchlet enery} in the case considered by Rochet and Chon\'e.

\subsection{When is this optimization problem convex?}
From \cite{MonteiroPage98} and \cite{Carlier01} we know that a minimizer exists.
The contribution of Figalli, Kim and McCann is to give sufficient conditions for
the variational problem to become convex --- in which case it is considerably
simpler to analyze, theoretically and computationally.
It is very interesting that the Ma, Trudinger and Wang criteria for the regularity of
optimal mappings turn out to be related to this question.
The following are among the main results of \cite{FigalliKimMcCann-econ}:

\begin{thm}[Convexity of the principal's strategy space \cite{FigalliKimMcCann-econ}]\label{T:convexity of principal's strategy space}
If $b=(-c)$ satisfies \Azero--\Atwo\ and \Afour\ then the set
$\mathcal{U} = \{u=u^{\tb b}\}$ is convex if and only if
\Bthree\ holds, ie., if and only if $\cross(p,q)\geq 0$ for all tangent vectors
$(p,x_0) \in TM^+$ and $(q,y_0) \in TM^-$.
\end{thm}

\begin{rem}
It was pointed out subsequently by Brendan Pass \cite{Pass-personal} that the
convexity of $M^-_{x_0}$ assumed in \Afour\ for each $x_0 \in M^+$
is also necessary for convexity of $\mathcal{U}$.
\end{rem}

\proof[Sketch of proof]
First assume \Bthree\ holds --- assuming always \Azero--\Atwo\ and \Afour.
Given $u_0,u_1 \in \mathcal{U}$ and $t \in [0,1]$ we claim $u_t := (1-t) u_0 + tu_1$ is $b$-convex.
This can be established by finding for each $x_0 \in M^+$ a $y_t \in M^-$ such that
\begin{equation}\label{b-support}
u_t(\cdot) \ge b(\cdot, y_t) - b(x_0, y_t) + u_t(x_0) \qquad {\rm throughout}\ M^+,
\end{equation}
for then $u_t(\cdot)$ is the supremum 
of such functions.
Corresponding to $t=0,1$ the desired points $y_0,y_1 \in M^-$ exist, by $b$-convexity of
$u_{i} = u_{i}^{\tb b}$ for $i=0,1$.
By \Afour, we can solve the equation $D_x b(x_0,y_t) = (1-t) D_x b(x_0,y_0) + t D_x b(x_0, y_t)$;
the solution $y_t \in M^-$ makes $f(\cdot, t) := b(\cdot, y_t) - b(x_0,y_t)$ a
convex function of $t\in [0,1]$,
according to Remark~\ref{R:convex DASM}.  Inequality
\eqref{b-support} holds at the endpoints $t=0,1$;  taking a convex combination yields
the desired inequality for intermediate values of $t \in [0,1]$.
For the converse direction, we refer to \cite{FigalliKimMcCann-econ}.
\endproof

\begin{thm}[Convexity of principal's losses and uniqueness of optimal strategy
 \cite{FigalliKimMcCann-econ}]\label{T:uniqueness of principal's losses}
If \Azero--\Afour\ and \Bthree\ hold and if $a=a^{b \tb}$, then the functional
$u \in \mathcal{U} \longmapsto L(u)$ defined by \eqref{principal's losses} is convex.
Furthermore, it has enough strict convexity to conclude the optimizer
$u \in \mathcal{U}_\emptyset$ is uniquely
determined (at least $\mu^+$-a.e.) if $\mu^+ \ll \HM^n$ and either
(i) $y_{b,a}:  M^+ \to M^-$ is continuous or else
(ii) $b$ has positive cross-curvature \Bthree$_s$.
\end{thm}

\proof[Sketch of proof]
To deduce convexity of $L: \mathcal{U} \longrightarrow \R$,  recall $\tb$-convexity of $a$
implies
$$
a(Y(x,p)) - b(x,Y(x,p)) = \sup_{x_1 \in M^+} b(x_1,Y(x,p)) - b(x,Y(x,p)) - a^b(x_1).
$$
For each $x \in M^+$ fixed,  the functions under the supremum are convex with respect
to $p \in M^-_x$, according to Remark \ref{R:convex DASM}.  Thus the integrand in
\eqref{principal's losses} is linear in $u(x)$ and convex with respect to $p = Du(x)$,
which establishes the desired convexity of the integral $L(u)$.  In case (i) the integrand
is strictly convex,  while in case (ii) it is strictly convex for all $x \in \dom Da^b$,
which is a set of full $\mu^+ \ll \Hn$ measure.
We refer to \cite{FigalliKimMcCann-econ} for details.
\endproof

Regarding robustness: we may mention that, as in Remark \ref{R:bilinear borderline},
the bilinear function $b(x,y)$ 
lies on the borderline of costs
which satisfy \Bthree.  Thus there will be perturbations of this function which destroy
convexity of the problem,  and we can anticipate that under such perturbations, uniqueness
and other properties of its solution may no longer persist.  In fact,  for $a=0$ and
$b(x,y) = -d^2_M(x,y)$ on a Riemannian ball $M^+=M^-=B_r(y_\emptyset)$,
we arrive at a problem equivalent to a fourfold symmetrized version of
 Rochet and Chon\'e's in the Euclidean case,
but which satisfies or violates \Bthree\ depending on whether the metric is spherical or hyperbolic.
This can used to model local delivery of a centralized resource for
a town in the mountains \cite{FigalliKimMcCann-econ}; cf.\ Examples \ref{E:sphere}--\ref{E:saddle}.
On the other hand,
under the hypotheses of Theorem~\ref{T:uniqueness of principal's losses}
we are able to show that Armstrong's ``desirability of exclusion'' \cite{Armstrong96}
continues to hold. We give the statement only and refer to \cite{FigalliKimMcCann-econ} for
its proof.

\begin{thm}[The desirability of exclusion]\label{T:desirability of exclusion}
Assume \Azero--\Afour, \Bthree, $a=a^{b \tb}$ and that $d\mu^+=f^+d\Hn$ with
$f^+ \in W^{1,1}(M^+)$ and the convex set $M^-_{y_\emptyset}= D_y b(M^+,y_\emptyset) \subset \R^n$
has no $(n-1)$ dimensional facets.  Then a positive fraction of agents will be priced out of the
market by the principal's optimal strategy.
\end{thm}

\begin{rem}
It is interesting to note that the strict convexity condition on $M^+_{y_\emptyset}$ holds neither
in one dimension --- where Armstrong noted counterexamples to the desirability of exclusion ---
nor for the example of Rochet-Chon\'e,  where convexity of $M^+_{y_\emptyset}$ is not strict.
\end{rem}

\subsection{Variant: maximizing social welfare}

Idealistic readers may be taken somewhat aback by the model just presented,  for it is the very
theory which predicts, among other things, just how uncomfortable airlines ought to
make their economy seating to ensure
--- without sacrificing too much economy-class revenue ---
that passengers with the means to secure a
business-class ticket have sufficient incentive to do so.
Such readers will doubtless be glad to know that the same mathematics is equally
relevant to the more egalitarian question of how to price public services so as to
maximize societal benefit.

For example, suppose the welfare $w(x,u(x))$ of agent $x$ is a concave function
of the indirect utility $u(x)$ he receives.  A public service provider would like to
set a price menu for which $u = v^b$ maximizes the toral welfare among all agents:
$$
\max \limits_{u \in \mathcal{U}_\emptyset,\; L(u)\leq 0} \int_{M^+} w(x,u(x))d\mu^+(x),
$$
subject to the constraint  $L(u) \le 0$ that the service provider not sustain losses.
Introducing a Lagrange multiplier $\lambda$ for this budget constraint, the problem
can be rewritten in the unconstrained form
$$
\max \limits_{u \in \mathcal{U}_\emptyset} -\lambda L(u)+\int_{M^+} w(x,u(x))d\mu^+(x),
$$
(for a suitable $\lambda \geq 0$).
Under the same assumptions \Azero--\Afour, \Bthree\ and $a=a^{b\tb}$ as before,
we see from the results above that this becomes a concave maximization problem for which existence and
uniqueness of solution follow directly,  and which is therefore quite amenable to further study,
both theoretical and computational.

\section{A pseudo-Riemannian and symplectic geometric afterword}
\label{S:differential geometry}

The conditions of Ma, Trudinger and Wang for regularity of optimal transport have
a differential geometric significance uncovered by Kim and McCann \cite{KimMcCann07p},
which led to their discovery with Warren \cite{KimMcCannWarren09p} of a surprising
connection of optimal transport to the theory of volume-maximizing
special Lagrangian submanifolds in split geometries.

Indeed, since the smoothness of optimal maps $G:M^+ \longrightarrow M^-$ is a question
whose answer is independent of coordinates chosen on $M^+$ and $M^-$,  it follows that
the necessary and sufficient condition \Athree\ for continuity in Theorem \ref{T:Loeper}
should have a geometrically invariant description.  We give this description, below, as
the positivity of certain sectional curvatures of a metric tensor $h$ induced on the
product manifold $N := M^+ \times M^-$ by the cost function $c \in C^4(N)$.  This motivates
the appellation {\em cross-curvature}. The rationale for such a description to exist is quite
analogous to that underlying  general relativity, Einstein's theory of gravity, which
can be expressed in the language of pseudo-Riemannian geometry due to the
coordinate invariance that results from the equivalence principle (observer independence).

Use the cost function to define the symmetric and antisymmetric tensors
\begin{equation}\label{pseudo-metric}
h = \sum_{i=1}^n \sum_{j=1}^n
\frac{\partial^2 c}{\partial x^i \partial y^j} (dx^i \otimes dy^j + d y^j \otimes d x^i)
\end{equation}
\begin{equation}\label{symplectic form}
\omega = \sum_{i=1}^n \sum_{j=1}^n
\frac{\partial^2 c}{\partial x^i \partial y^j} (dx^i \otimes dy^j - d y^j \otimes d x^i)
\end{equation}
on $N=M^+ \times M^-$.  Then condition \Atwo\ is equivalent to non-degeneracy of the
metric tensor $h$,  which in turn is equivalent to the assertion that $\omega$ is a
symplectic form.  Note however that $h$ is not positive-definite,  but has equal numbers
of positive and negative eigenvalues in any chosen coordinates, ie. signature
$(n,n)$.  Conditions \Athree\ and \Afour\ are conveniently re-expressed in terms of
the pseudo-metric $h$,  and its pseudo-Riemannian curvature tensor $R_{ijkl}$
\cite{KimMcCann07p}.  Indeed,
condition \Afour\ asserts the $h$-geodesic convexity of
$\{x_0\} \times M^-$ and $M^+ \times \{y_0\}$,  while the formula
$$
\cross(p,q) = R_{ijkl} p^i q^j p^k q^l
$$
shows the cross-curvature is simply proportional to the pseudo-Riemannian sectional curvature
of the 2-plane $(p \oplus 0)\wedge (0 \oplus q)$.  The restriction distinguishing \Athree\ from
\Bthree\ is that $p \oplus q$ be lightlike,  which is equivalent to the $h$-orthogonality
of $p \oplus 0$ and $0 \oplus q$.

Kim and McCann went on to point out that the graph of any $c$-optimal map is
$h$-spacelike and $\omega$-Lagrangian,  meaning any tangent vectors $P,Q \in T_{(x,G(x))}N$
to this graph satisfy $h(P,P) \ge 0$ and $\omega(P,Q)=0$.  This is a consequence of
Corollary \ref{C:foc}, it is also very illustrative to check it directly by hand in the case of the quadratic cost in $\mathbb{R}^n$.
When \Azero--\Afour\ hold they also showed the converse to be true:
any diffeomorphism whose graph is $h$-spacelike and $\omega$-Lagrangian is also $c$-cyclically
monotone, hence optimal.  With Warren \cite{KimMcCannWarren09p},
they introduced a pseudo-metric
\begin{equation}\label{conformal metric}
h^{f^\pm}_c = \left(\frac{f^+(x) f^-(y)}{|\det c_{i,j}(x,y) |} \right)^{1/n} h
\end{equation}
conformally equivalent to $h$.  In this new metric, they show the graph of the $c$-optimal map
pushing $d\mu^+(x) = f^+(x) d^n x$ forward to $\mu^-(y) = f^-(y)d^n y$
has maximal volume with respect to compactly supported perturbations.
In particular, $Graph(G)$ has zero
mean-curvature as a submanifold (with half the dimension) of $(N,h^{f^\pm}_c)$ ---
yielding an unexpected connection of optimal transportation to more classical problems in
geometry and geometric measure theory.  (Note that the
metric \eqref{conformal metric} depends only on the measures $f^\pm$
and the sign of the mixed partial $D^2_{xy} c$ in dimension $n=1$.)
The preprint  of Harvey and Lawson \cite{HarveyLawson10p}
contains a wealth of related information concerning special Lagrangian
submanifolds in pseudo-Riemannian (= semi-Riemannian) geometry.


\end{document}